\documentclass[12pt, reqno]{amsart}

\usepackage{amsmath, amsthm, amssymb}
\usepackage{enumerate}
\usepackage[margin=1.0in]{geometry}
\usepackage{xcolor}
\definecolor{cite}{rgb}{0.30,0.60,1.00}
\definecolor{url}{rgb}{0.00,0.00,0.80}
\definecolor{link}{rgb}{0.40,0.10,0.20}
\usepackage[pdfusetitle,colorlinks,linkcolor=link,urlcolor=url,citecolor=cite,pagebackref,breaklinks]{hyperref}
\usepackage{graphicx}
\usepackage{cleveref}
\usepackage{mathdots}
\usepackage{tikz-cd}
\usepackage{comment}
\usepackage{xypic}
\usepackage{mathtools}

\newtheorem{theorem}{Theorem}[section]

\newtheorem{proposition}[theorem]{Proposition}

\newtheorem{corollary}[theorem]{Corollary}

\theoremstyle{definition}

\theoremstyle{definition}

\theoremstyle{definition}

\newcommand{\cComplex}{\mathbb{C}}
\newcommand{\multiplicativegroup}[1]{#1^{\times}}

\newcommand{\Hom}{\mathrm{Hom}}

\newcommand{\idmap}{\mathrm{id}}
\newcommand{\conjugate}[1]{\overline{#1}}

\newcommand{\abs}[1]{\left|#1\right|}
\newcommand{\sizeof}[1]{\left|#1\right|}

\newcommand{\innerproduct}[2]{\left\langle #1,#2\right\rangle}

\newcommand{\fieldCharacter}{\psi}

\newcommand{\Ind}[3]{\mathrm{Ind}_{#1}^{#2}\left(#3\right)}

\newcommand{\SpehRepresentation}[2]{\Delta\left(#1, #2\right)}

\newcommand{\besselSpehFunction}[2]{\mathcal{BS}_{\SpehRepresentation{#1}{#2}, \fieldCharacter}}

\newcommand{\Irr}{\mathrm{Irr}}

\newcommand{\IdentityMatrix}[1]{I_{#1}}
\newcommand{\diag}{\mathrm{diag}}

\newcommand{\trace}{\operatorname{tr}}
\newcommand{\GL}{\mathrm{GL}}

\newcommand{\UnipotentRadicalForWssRecursion}[2]{\mathcal{Y}_{c,k}}

\newcommand{\FieldNorm}[2]{\mathrm{N}_{#1:#2}}

\newcommand{\FieldTrace}{\mathrm{Tr}}
\newcommand{\finiteField}{\mathbb{F}}
\newcommand{\finiteFieldExtension}[1]{\finiteField_{#1}}

\newcommand{\algebraicClosure}[1]{\overline{#1}}
\newcommand{\charactergroup}[1]{\widehat{\multiplicativegroup{\finiteFieldExtension{#1}}}}

\newcommand{\Frobenius}{\operatorname{Fr}}

\newcommand{\multiplcativeScheme}{\mathbb{G}_m}
\newcommand{\affineLine}{\mathbb{A}^1}
\newcommand{\squareMatrix}{\operatorname{Mat}}

\newcommand{\SymmetricGroup}{\mathfrak{S}}

\newcommand{\Erdelyi}{Erd{\'e}lyi}
\newcommand{\Toth}{T{\'o}th}

\newcommand{\GKGaussSum}[3]{\mathcal{G}\left(#1 \times #2, #3\right)}
\newcommand{\GKGaussSumScalar}[3]{\mathrm{G}\left(#1 \times #2, #3\right)}

\newcommand{\ExoticKloosterman}{\mathrm{Kl}}
\newcommand{\GaussSumCharacter}[4]{\tau_{#1}\left(#2 \times #3, #4\right)}
\newcommand{\IrrCusp}{\Irr_{\mathrm{cusp}}}
\newcommand{\convolutionWithCompactSupport}{\boldsymbol{\mathrm{R}}}
\newcommand{\ladicnumbers}{\algebraicClosure{\mathbb{Q}_{\ell}}}
\newcommand{\artinScrier}{\operatorname{AS}}
\newcommand{\KloostermanSumClassFunction}{\mathcal{K}}

\hypersetup{pdfauthor={Elad Zelingher},
	pdfsubject={Number theory, Representation theory},
	pdfkeywords={Kloosterman sums, Gauss sums}}

\title{On matrix Kloosterman sums and Hall--Littlewood polynomials}

\author{Elad Zelingher}
\address{Department of Mathematics, University of Michigan, 1844 East Hall, 530 Church Street, Ann Arbor, MI 48109-1043 USA}
\email{eladz@umich.edu}


\begin{document}

\begin{abstract}
	We prove an identity relating twisted matrix Kloosterman sums to modified Hall--Littlewood polynomials evaluated at the roots of the characteristic polynomial associated to a twisted Kloosterman sheaf. This solves a conjecture of \Erdelyi{} and \Toth{} \cite[Conjecture 5.9]{erdelyi2021matrix}.
\end{abstract}
	\dedicatory{\bf In memory of Jonathan Seidman}
\maketitle

\section{Introduction}\label{sec:introduction}
Kloosterman sums are central objects in number theory. They are used in analytic number theory \cite{KowalskiMichelSawin2017}, play an important role in the Kuznetsov trace formula \cite{DeshouillersIwaniec1982}, and appear in formulas in representation theory \cite{Zelingher2023}. Kloosterman sheaves are sheaf incarnations of Kloosterman sums. They were defined by Deligne \cite{deligne569cohomologie} and were studied extensively by Katz \cite{katz2016gauss}. Kloosterman sheaves play an important role in the geometric Langlands program, and provide a source of automorphic representations for the function field case  \cite{HeinlothNgoYun2013, Yun2016}.

Let $\finiteField$ be a finite field with $q$ elements, and let $\fieldCharacter \colon \finiteField \to \multiplicativegroup{\cComplex}$ be a non-trivial additive character. Recently, \Erdelyi{} and \Toth{} started studying \emph{matrix Kloosterman sums}, see \cite{erdelyi2021matrix, ErdelyiTothZabradi2024, erdelyi2022purity}. For any $x \in \squareMatrix_n\left(\finiteField\right)$, they considered the family $$\ExoticKloosterman_m\left(\fieldCharacter, x\right) = \sum_{g \in \GL_n\left(\finiteField_m\right)} \fieldCharacter_m\left(\trace \left(g + x g^{-1}\right)\right),$$
where $\finiteFieldExtension{m} \slash \finiteField$ is a field extension of degree $m$, and $\fieldCharacter_m = \fieldCharacter \circ \trace_{\finiteFieldExtension{m} \slash \finiteField}$. Write $\ExoticKloosterman\left(\fieldCharacter, x\right)$ for $\ExoticKloosterman_1\left(\fieldCharacter, x\right)$. Notice that the sum $\ExoticKloosterman_m\left(\fieldCharacter, x\right)$ does not depend on the conjugacy class of $x$. Let us mention that these sums were considered and studied previously for some special cases, especially for the case where $x$ is a scalar matrix. They were considered in a slightly more general form by Hodges in \cite{Hodges1956}, and were evaluated for a scalar matrix $x$ in \cite[Theorem C]{Kim1998}, \cite[Theorem 2]{ChaeKim2003}, and \cite[Theorem 2]{Fulman2001}. See also \cite[Section 1.4.3]{GorodetskyRodgers2021}.

\Erdelyi{} and \Toth{} proved many properties of this family, some of which are analogous to the properties of classical Kloosterman sums. Among these properties, they proved a multiplicativity property: if $n = n_1 + n_2$ and $x_1 \in \squareMatrix_{n_1}\left(\finiteField\right)$ and $x_2 \in \squareMatrix_{n_2}\left(\finiteField\right)$ have no common eigenvalues over the algebraic closure $\algebraicClosure{\finiteField}$, then $$\ExoticKloosterman\left(\fieldCharacter, \diag\left(x_1, x_2\right)\right) = q^{n_1 n_2} \ExoticKloosterman\left(\fieldCharacter, x_1\right) \ExoticKloosterman\left(\fieldCharacter, x_2\right).$$
They also computed $\ExoticKloosterman\left(\fieldCharacter, x\right)$ for nilpotent elements $x$ and gave an algorithm for computing $\ExoticKloosterman\left(\fieldCharacter, x\right)$ in terms of a polynomial in the classical Kloosterman sum $\ExoticKloosterman\left(\fieldCharacter, \xi \right)$ where $x$ is a matrix with a single eigenvalue $\xi \in \finiteField$. 

Since $\ExoticKloosterman\left(\fieldCharacter, x\right)$ only depends on the conjugacy class of $x$, thanks to the multiplicativity property, it suffices to know how to compute $\ExoticKloosterman\left(\fieldCharacter, x\right)$ for a matrix $x$ whose characteristic polynomial is a power of an irreducible polynomial over $\finiteField$. The simplest example that has not been treated by the work of \Erdelyi{} and \Toth{} above is the case where $x$ has an irreducible characteristic polynomial over $\finiteField$. Such $x$ is called a \emph{regular elliptic element of $\GL_n\left(\finiteField\right)$}. Given such $x$, let $\{\xi, \xi^q, \dots, \xi^{q^{n-1}}\}$ be its $n$ different eigenvalues in $\finiteFieldExtension{n}$. Then \Erdelyi{} and \Toth{} conjectured \cite[Conjecture 5.9]{erdelyi2021matrix} that if $q = p^f$ and $p$ is a large enough prime, then $$\ExoticKloosterman\left(\fieldCharacter, x\right) = \left(-1\right)^{n+1} q^{\binom{n}{2}} \ExoticKloosterman_n\left(\fieldCharacter, \xi\right).$$
They gave a heuristic explanation for this conjecture, performed computations that provided evidence of it, and proved it for the special case where $n = 2$.

In this work, we prove identities confirming this conjecture in
a more general form. Our first theorem (\Cref{thm:regular-elliptic-elements}) gives a generalization of the conjecture of \Erdelyi{} and \Toth{} for regular elliptic elements.

\begin{theorem}\label{thm:main-theorem}
	For any finite field $\finiteField$ and any regular elliptic $x \in \GL_n\left(\finiteField\right)$ with eigenvalues $\{\xi, \xi^q, \dots, \xi^{q^{n-1}}\}$, and for any $k > 1$, we have the identity
	$$ \sum_{\substack{g_1, \dots, g_k \in \GL_n\left(\finiteField\right)\\
	g_1 \cdot g_2 \cdot \dots \cdot g_k = x}} \fieldCharacter\left(\trace \sum_{i=1}^k g_i \right) = \left(-1\right)^{\left(k-1\right)\left(n+1\right)} q^{\left(k-1\right)\binom{n}{2}} \sum_{\substack{t_1,\dots,t_k \in \multiplicativegroup{\finiteFieldExtension{n}}\\
\prod_{i=1}^k t_i = \xi}} \fieldCharacter_n\left( \sum_{i=1}^k t_i \right).$$
More generally, for any character $\alpha \colon \left(\multiplicativegroup{\finiteField}\right)^k \to \multiplicativegroup{\cComplex}$, we have $$ \ExoticKloosterman\left(\alpha, \fieldCharacter, x\right) = \left(-1\right)^{\left(k - 1\right)\left(n + 1\right)} q^{\left(k-1\right)\binom{n}{2}} \ExoticKloosterman_n\left(\alpha, \fieldCharacter, \xi\right),$$
where $\ExoticKloosterman\left(\alpha, \fieldCharacter, x\right)$ and $\ExoticKloosterman_n\left(\alpha, \fieldCharacter, \xi\right)$ are the twisted matrix Kloosterman sum and the twisted Kloosterman sum, respectively (see \Cref{sec:twisted-kloosterman-sums}).
\end{theorem}

Our method of proving this theorem is based on the representation theory of finite general linear groups. We utilize a computation of non-abelian Gauss sums due to Kondo \cite{Kondo1963}. We then use a classification for representations of $\GL_n\left(\finiteField\right)$ whose characters support regular elliptic elements, and a formula for their character values at these elements. This was worked out by Graham Gordon \cite{Gordon2022}.

Our second theorem (\Cref{thm:explicit-expression-for-jordan-block-with-hall-littlewood-polynomial}) allows us to express twisted matrix Kloosterman sums of generalized Jordan matrices as modified Hall--Littlewood polynomials, evaluated at roots of the characteristic polynomial associated with the action of the geometric Frobenius on a stalk of a twisted Kloosterman sheaf. This generalizes \cite[Theorem 3]{erdelyi2022purity}, where it was proved for the special case $a = 1$, $k = 2$ and $\alpha = 1$.
\begin{theorem}\label{thm:main-thm-hall-littlewood}
	For any finite field $\finiteField$, any $n = ab$, any regular elliptic $x \in \GL_a\left(\finiteField\right)$ with eigenvalues $\{\xi, \xi^q, \dots, \xi^{q^{a-1}}\}$, and any partition $\mu \vdash b$, we have the identity
	$$ \sum_{\substack{g_1, \dots, g_k \in \GL_n\left(\finiteField\right)\\
			g_1 \cdot g_2 \cdot \dots \cdot g_k = J_{\mu}\left(x\right)}} \fieldCharacter\left(\trace \sum_{i=1}^k g_i \right) = \left(-1\right)^{\left(k-1\right)n} q^{\left(k-1\right)\binom{n}{2}} \mathrm{\tilde{H}}_{\mu}\left(\omega_1,\dots,\omega_k;q^a\right),$$
		where $J_{\mu}\left(x\right)$ is the generalized Jordan matrix corresponding to $\mu$ and $x$ (see \Cref{subsec:conjugacy-classes-of-gln}), the complex numbers $\omega_1, \dots, \omega_k$ are the roots of $\ExoticKloosterman^{\finiteFieldExtension{a}}\left( \fieldCharacter \right) = \ExoticKloosterman^{\finiteFieldExtension{a}}\left(1, \fieldCharacter \right)$ at $\xi$ (see \Cref{subsec:roots-of-frobenius}) and $\mathrm{\tilde{H}}_{\mu}\left(x_1,\dots,x_k;t\right)$ is the modified Hall--Littlewood polynomial (see \Cref{subsec:hall-littlewood}): $$\mathrm{\tilde{H}}_{\mu}\left(x_1,\dots,x_k;q^a\right) = \sum_{\lambda \vdash b}  \# \left\{ \mathcal{F} \text{ flag of type } \lambda \text{ in } \finiteFieldExtension{a}^b \mid J_{\mu}\left(\xi\right) \mathcal{F} = \mathcal{F}  \right\} \cdot m_{\lambda}\left(x_1,\dots,x_k\right).$$ Here, $m_{\lambda}\left(x_1,\dots,x_k\right)$ is the monomial symmetric polynomial corresponding to the partition $\lambda$. More generally, for any character $\alpha \colon \left(\multiplicativegroup{\finiteField}\right)^k \to \multiplicativegroup{\cComplex}$, we have
		$$ \ExoticKloosterman\left(\alpha, \fieldCharacter, J_{\mu}\left(x\right) \right) = \left(-1\right)^{\left(k-1\right)n} q^{\left(k-1\right)\binom{n}{2}} \mathrm{\tilde{H}}_{\mu}\left(\omega_1,\dots,\omega_k;q^a\right),$$
		where $\omega_1, \dots, \omega_k$ are the roots of $\ExoticKloosterman^{\finiteFieldExtension{a}}\left(\alpha, \fieldCharacter \right)$ at $\xi$ (see \Cref{subsec:roots-of-frobenius}).
\end{theorem}
Given this result, the polynomials studied by \Erdelyi{} and \Toth{} are the well studied modified Hall--Littlewood polynomials\footnote{Expressed in terms of Dickson polynomials, see \cite[Section 3]{curtis1999unitary} or \cite[Section 4.4]{Zelingher2023}.}, and the recursive algorithm that \Erdelyi{} and \Toth{} give for these polynomials is similar to algorithms for computing versions of Hall--Littlewood polynomials. Compare for example \cite[Theorem 5.1]{erdelyi2021matrix} and \cite[Formula (18)]{DesarmenienLeclercThibon1994}.

The proof of this theorem is similar to the proof of the previous one. We still make use of Kondo's formula, but this time we do not use the basis of characters of irreducible representations of $\GL_n\left(\finiteField\right)$. Instead, we use a different basis for the space of class functions on $\GL_n\left(\finiteField\right)$. This basis consists of Green's \emph{basic characters} \cite{Green55}, and it is more convenient for our application. Although \Cref{thm:main-thm-hall-littlewood} is a generalization of \Cref{thm:main-theorem}, we decided to keep the proof of \Cref{thm:main-theorem}, as it is simpler, and it allows the reader who is interested exclusively in the conjecture of \Erdelyi{} and \Toth{} to read a short proof of it. Another reason for leaving both proofs is to demonstrate that even though it is more natural to use the basis of characters of irreducible representations for the space of class functions of $\GL_n\left(\finiteField\right)$, other bases are sometimes good for certain applications.

In the special case where $\mu = \left(b\right)$, our result expresses matrix Kloosterman sums of regular Jordan matrices as traces of symmetric powers of twisted Kloosterman sheaves. The result in this special case is presented in the following theorem (\Cref{cor:symmetric-powers}). This was shown by \Erdelyi{} and \Toth{} in \cite[Section 5E]{erdelyi2021matrix} for the special case $a = 1$, $k = 2$ and $\alpha = 1$.
\begin{theorem}\label{thm:main-thm-symmetric-powers}
	For any finite field $\finiteField$, any character $\alpha \colon \left(\multiplicativegroup{\finiteField}\right)^k \to \multiplicativegroup{\cComplex}$, any $n = ab$, and any regular elliptic $x \in \GL_a\left(\finiteField\right)$ with eigenvalues $\{\xi, \xi^q, \dots, \xi^{q^{a-1}}\}$, we have the identity
	$$ \ExoticKloosterman\left(\alpha, \fieldCharacter, J_{\left(b\right)}\left(x\right) \right) = \left(-1\right)^{\left(k-1\right)n} q^{\left(k-1\right)\binom{n}{2}} \trace\left(\Frobenius\mid_{\xi}, \operatorname{Sym}^b \ExoticKloosterman^{\finiteFieldExtension{a}}\left(\alpha, \fieldCharacter\right)_{\xi} \right),$$
	where $\ExoticKloosterman^{\finiteFieldExtension{a}}\left(\alpha, \fieldCharacter\right)$ is the twisted Kloosterman sheaf defined over $\finiteFieldExtension{a}$, associated with $\alpha$ and $\fieldCharacter$ (see \Cref{subsec:roots-of-frobenius}).
\end{theorem}

We use our results and the purity results of twisted Kloosterman sheaves to obtain good bounds for twisted matrix Kloosterman sums (\Cref{cor:upper-bound}).
\begin{theorem}\label{thm:upper-bound}
	For any finite field $\finiteField$, any character $\alpha \colon \left(\multiplicativegroup{\finiteField}\right)^k \to \multiplicativegroup{\cComplex}$, any $n = ab$, any regular elliptic $x \in \GL_a\left(\finiteField\right)$, and any $\mu \vdash b$, we have
	$$ \abs{\ExoticKloosterman\left(\alpha, \fieldCharacter, J_\mu\left(x\right) \right)} \le q^{\left(k-1\right) \frac{n^2}{2}} \cdot \#\left\{\mathcal{F} \text{ is a weak flag of length } k \text{ in } \finiteFieldExtension{a}^b \mid J_{\mu}\left(x\right) \mathcal{F} = \mathcal{F} \right\}.$$
	In the special case $\mu = \left(b\right)$, this result reads
	$$ \abs{\ExoticKloosterman\left(\alpha, \fieldCharacter, J_{\left(b\right)}\left(x\right) \right)} \le q^{\left(k-1\right) \frac{n^2}{2}} \binom{b + k -1}{b}.$$
\end{theorem}

The idea for the proofs of \Cref{thm:main-theorem} and \Cref{thm:main-thm-hall-littlewood} comes from a deeper idea in representation theory of $\GL_n\left(\finiteField\right)$. Twisted matrix Kloosterman sums are class functions of $\GL_n\left(\finiteField\right)$, so they can be written as a linear combination of characters of irreducible representations of $\GL_n\left(\finiteField\right)$. What is the coefficient in this linear combination of $\trace \pi$, for an irreducible representation $\pi$? It turns out that the answer is closely related to the tensor product representation gamma factor associated to $\pi$ and the principal series representation associated with the character $\alpha \colon \left(\multiplicativegroup{\finiteField}\right)^k \to \multiplicativegroup{\cComplex}$.

More generally, in a recent work with Oded Carmon \cite{CarmonZelingher2024}, we encountered exotic matrix Kloosterman sums while studying a finite field analog of Ginzburg--Kaplan gamma factors \cite[Appendix A]{kaplan2018}. Our results in \cite{CarmonZelingher2024} translate between properties of gamma factors and properties of matrix Kloosterman sums. For example, the multiplicativity property of gamma factors is closely related to the multiplicativity property of matrix Kloosterman sums discussed above. In the appendix, we explain one of our results from \cite{CarmonZelingher2024}, that relates matrix Kloosterman sums to Speh representations.

Since our results in \cite{CarmonZelingher2024} are very representation theoretic focused, and since researchers with interest in Kloosterman sums are not necessarily interested in representation theory, we decided to write this separate work with the minimum knowledge required to establish these identities. However, throughout the text we remark that using a multiplicativity result we proved in \cite{CarmonZelingher2024} (proved in \cite[Theorem 1.1]{erdelyi2021matrix} for the special case $k = 2$ and $\alpha = 1$), we are able to use our results for Jordan matrices mentioned above to express any matrix Kloosterman sum as a product of modified Hall--Littlewood polynomials evaluated at certain points. We also take this opportunity to write an exposition on the parallelism between conjugacy classes of $\GL_n\left(\finiteField\right)$ and the representation theory of $\GL_n\left(\finiteField\right)$, which is, of course, known to experts.

\subsection{Acknowledgments}

I would like to thank Oded Carmon for many discussions about Speh representations and matrix Kloosterman sums. I would also like to thank Ofir Gorodetsky for encouraging me to write this paper. Finally, I would like to thank the anonymous referee for their careful reading of this manuscript and their comments.

\section{Preliminaries}\label{sec:preliminaries}
Let $\finiteField$ be a finite field with cardinality $q$. Fix an algebraic closure $\algebraicClosure{\finiteField}$, and for any $n \ge 1$, let $\finiteFieldExtension{n}$ be the unique extension of $\finiteField$ of degree $n$ in $\algebraicClosure{\finiteField}$. Let $\charactergroup{n}$ be the character group of $\multiplicativegroup{\finiteFieldExtension{n}}$, consisting of characters $\theta \colon \multiplicativegroup{\finiteFieldExtension{n}} \to \multiplicativegroup{\cComplex}$. If $a \mid n$, let $\FieldNorm{n}{a} \colon \multiplicativegroup{\finiteFieldExtension{n}} \to \multiplicativegroup{\finiteFieldExtension{a}}$ be the norm map. If $\fieldCharacter \colon \finiteField \to \multiplicativegroup{\cComplex}$ is an additive character, we denote $\fieldCharacter_n = \fieldCharacter \circ \FieldTrace_{\finiteFieldExtension{n} \slash \finiteField}$.

\subsection{Twisted Kloosterman sums}\label{sec:twisted-kloosterman-sums}
Let $\fieldCharacter \colon \finiteField \to \multiplicativegroup{\cComplex}$ be a non-trivial additive character. Let $\alpha = \left(\multiplicativegroup{\finiteField}\right)^k \to \multiplicativegroup{\cComplex}$ be a character. Write $\alpha 
 = \alpha_1 \times \dots \times \alpha_k$, where $\alpha_1, \dots, \alpha_k \colon \multiplicativegroup{\finiteField} \to \multiplicativegroup{\cComplex}$ are characters.
For any $n \ge 1$, and any $\xi \in \multiplicativegroup{\finiteFieldExtension{n}}$, we define the \emph{twisted Kloosterman sum} $$\ExoticKloosterman_n\left(\alpha, \fieldCharacter, \xi\right) = \sum_{\substack{t_1,\dots,t_k \in \multiplicativegroup{\finiteFieldExtension{n}}\\
		\prod_{j=1}^k t_j = \xi}} \left(\prod_{i=1}^k \alpha_i \left(\FieldNorm{n}{1}\left(t_i \right)\right)\right) \fieldCharacter_n \left(\sum_{i=1}^k t_i\right).$$
	Notice that this sum is constant on Frobenius orbits of $\multiplicativegroup{\finiteFieldExtension{n}}$, i.e., $\ExoticKloosterman_n\left(\alpha, \fieldCharacter, \xi \right) = \ExoticKloosterman_n\left(\alpha, \fieldCharacter, \xi^q \right)$, as the trace and the norm maps are constant on Frobenius orbits.
	
	For any $m \ge 1$, and any $x \in \GL_m\left(\finiteField\right)$, we define the \emph{twisted matrix Kloosterman sum} $$\ExoticKloosterman \left(\alpha, \fieldCharacter, x\right) = \sum_{\substack{g_1,\dots,g_k \in \GL_m\left(\finiteField\right)\\
g_1 \cdot g_2 \cdot \dots \cdot g_k = x}} \left(\prod_{i=1}^k \alpha_i \left(\det g_i \right)\right) \fieldCharacter \left(\trace \left(\sum_{i=1}^k g_i\right) \right).$$
Notice that if $x, y \in \GL_m\left(\finiteField\right)$ are conjugate, then $\ExoticKloosterman\left(\alpha, \fieldCharacter, x\right) = \ExoticKloosterman\left(\alpha, \fieldCharacter, y\right)$, as one can change variables in the sum, using the fact that trace and determinant are invariant under conjugation. This definition does not depend on the order of the characters $\alpha_1, \dots, \alpha_k$ either, as one can conjugate any variable by a product of the others. For example, by replacing $g_2$ by $g_1^{-1} h_2 g_1$, we get a sum over the condition $h_2 g_1 g_3 \dots g_k = x$. For a different example, by replacing $g_3$ by $g_2^{-1} g_1^{-1} h_3 g_1 g_2$, we get a sum over the condition $h_3 g_1 g_2 g_4 \dots g_k = x$. As a result, the definition is invariant under permutation of the variables $g_1, \dots, g_k$, and therefore we may write $$\ExoticKloosterman \left(\alpha, \fieldCharacter, x\right) = \sum_{\substack{g_1,\dots,g_k \in \GL_m\left(\finiteField\right)\\
		\prod_{j=1}^k g_j = x}} \left(\prod_{i=1}^k \alpha_i \left(\det g_i \right)\right) \fieldCharacter \left(\trace \left(\sum_{i=1}^k g_i\right) \right)$$
without specifying the order of the product.

In \cite{CarmonZelingher2024}, we prove the following multiplicativity property. It was proved in \cite[Theorem 1.1]{erdelyi2021matrix} for the special case $k = 2$ and $\alpha = 1$.
\begin{theorem}\label{thm:multiplicativity-property}
	Let $\alpha \colon \left(\multiplicativegroup{\finiteField}\right)^k \to \multiplicativegroup{\cComplex}$ be a character. If $n = n_1 + n_2$ and $x_1 \in \GL_{n_1}\left(\finiteField\right)$ and $x_2 \in \GL_{n_2}\left(\finiteField\right)$ do not have common eigenvalues over $\algebraicClosure{\finiteField}$, then $$\ExoticKloosterman\left(\alpha, \fieldCharacter, \diag\left(x_1, x_2\right)\right) = q^{\left(k-1\right)n_1 n_2} \ExoticKloosterman\left(\alpha, \fieldCharacter, x_1\right) \ExoticKloosterman\left(\alpha, \fieldCharacter, x_2\right).$$
\end{theorem}

\subsection{Conjugacy classes of $\GL_n\left(\finiteField\right)$}\label{subsec:conjugacy-classes-of-gln}

An element $x \in \GL_n\left(\finiteField\right)$ is called \emph{regular elliptic} if its characteristic polynomial is irreducible over $\finiteField$. This is equivalent to $x$ having eigenvalues $\{ \xi, \xi^q, \dots, \xi^{q^{n-1}} \}$, where $\xi \in \multiplicativegroup{\finiteFieldExtension{n}}$ and $\xi^{q^j} \ne \xi$ for $1 \le j \le n-1$. Such $x$ is conjugate over the algebraic closure $\algebraicClosure{\finiteField}$, to $\diag(\xi, \xi^q, \dots, \xi^{q^{n-1}})$.

Given an element $y \in \GL_d\left(\finiteField\right)$, and $m \ge 1$, let $J_{\left(m\right)}\left(y\right)$ be the matrix $$
	J_{\left(m\right)}\left(y\right) = \begin{pmatrix}
		y & \IdentityMatrix{d}\\
		& y & \IdentityMatrix{d}\\
		& & \ddots & \ddots \\
		& & & y & \IdentityMatrix{d}\\
		& & & & y 
	\end{pmatrix} \in \GL_{dm}\left(\finiteField\right).$$
If $y \in \GL_d\left(\finiteField\right)$ is regular elliptic with eigenvalues $\{\eta,\eta^q,\dots,\eta^{q^{d-1}}\}$ in $\multiplicativegroup{\finiteFieldExtension{d}}$, then $J_{\left(m\right)}\left(y\right)$ is conjugate over the algebraic closure $\algebraicClosure{\finiteField}$ to $\diag(J_{\left(m\right)}^{\finiteFieldExtension{d}} \left(\eta\right),J_{\left(m\right)}^{\finiteFieldExtension{d}} \left(\eta^q\right),\dots,J_{\left(m\right)}^{\finiteFieldExtension{d}} (\eta^{q^{d-1}}))$, where
$$J_{\left(m\right)}^{\finiteFieldExtension{d}} \left(\eta\right) = \begin{pmatrix}
	\eta & 1\\
	& \eta & 1\\
	& & \ddots & \ddots\\
	& & & \eta & 1\\
	& & & & \eta
\end{pmatrix} \in \GL_{m}\left(\finiteFieldExtension{d}\right).$$

Recall that a \emph{partition} $\lambda = \left(m_1, \dots, m_r\right)$ is a finite weakly decreasing sequence of positive integers $m_1 \ge \dots \ge m_r$. We denote $\sizeof{\lambda} = \sum_{i=1}^r m_i$ and $\ell \left(\lambda\right) = r$. If $m$ is a non-negative integer, we write $\lambda \vdash m$ if $\sizeof{\lambda} = m$. We denote the empty partition of $0$ by $\left(\right)$.

Given a partition $\lambda = \left(m_1,\dots,m_r\right) \vdash m$, and $y \in \GL_d\left(\finiteField\right)$ we denote $J_{\lambda}\left(y\right) = \diag\left(J_{\left(m_1\right)}\left(y\right), \dots, J_{\left(m_r\right)}\left(y\right)\right) \in \GL_{dm}\left(\finiteField\right)$.

Recall that two elements in $\GL_n\left(\finiteField\right)$ are conjugate by a matrix in $\GL_n\left(\algebraicClosure{\finiteField}\right)$ if and only if they are conjugate by a matrix in $\GL_n\left(\finiteField\right)$. Thus, henceforth we will simply say that two elements in $\GL_n\left(\finiteField\right)$ are conjugate, without specifying over which field.

The classification of conjugacy classes of $\GL_n\left(\finiteField\right)$ is the result that for any element $g$, there exist $s \ge 1$, $d_1, \dots ,d_s \ge 1$, regular elliptic elements $y_1 \in \GL_{d_1}\left(\finiteField\right), \dots, y_s \in \GL_{d_s}\left(\finiteField\right)$, and non-empty partitions $\lambda_1, \dots, \lambda_s$, such that $g$ is conjugate to $\diag\left(J_{\lambda_1}\left(y_1\right), \dots, J_{\lambda_s}\left(y_s\right)\right)$, where $\sum_{j=1}^s d_j \sizeof{\lambda_j} = n$, and where $y_1, \dots, y_s$ have mutually disjoint eigenvalues over the algebraic closure $\algebraicClosure{\finiteField}$.
 
An element $x \in \GL_n\left(\finiteField\right)$ is called \emph{regular} if it is conjugate to an element $\diag\left(J_{\lambda_1}\left(y_1\right),\dots, J_{\lambda_s}\left(y_s\right)\right)$ as above, where $\lambda_1 = \left(m_1\right)$, $\dots$, $\lambda_s = \left(m_s\right)$, for $m_1, \dots, m_s \ge 1$, that is, if $x$ is conjugate to $\diag\left(J_{\left(m_1\right)}\left(y_1\right),\dots, J_{\left(m_s\right)}\left(y_s\right)\right)$. An element $x \in \GL_n\left(\finiteField\right)$ is called \emph{semisimple} if it is conjugate to an element $\diag\left(J_{\lambda_1}\left(y_1\right),\dots, J_{\lambda_s}\left(y_s\right)\right)$ as above, where $\lambda_1 = \left(1^{m_1}\right)$, $\dots$, $\lambda_s = \left(1^{m_s}\right)$, that is, if $x$ is conjugate to $\diag\left(y_1, \dots, y_1, \dots, y_s, \dots, y_s\right)$, where $y_j$ appears $m_j$ times.

\subsection{Representation theory of $\GL_n\left(\finiteField\right)$}\label{subsec:representation-theory-of-gln}

Let $\left(n_1, \dots, n_r\right)$ be a composition of $n$, that is, $n_1, \dots, n_r > 0$ and $n_1 + \dots + n_r = n$. We denote $$N_{\left(n_1,\dots,n_r\right)} = \left\{ \begin{pmatrix}
	\IdentityMatrix{n_1} & \ast & \ast & \ast \\
	& \IdentityMatrix{n_2} & \ast & \ast\\
	& & \ddots & \ast \\
	& & & \IdentityMatrix{n_r}
\end{pmatrix} \right\} \text{ and } D_{\left(n_1,\dots,n_r\right)} = \left\{ \diag\left(g_1,\dots,g_n\right) \mid g_j \in \GL_{n_j}\left(\finiteField\right) \right\}.$$
The parabolic subgroup $P_{\left(n_1,\dots,n_r\right)}$ of $\GL_n\left(\finiteField\right)$ corresponding to the composition $\left(n_1,\dots,n_r\right)$ is defined to be the semi-direct product $P_{\left(n_1,\dots,n_r\right)} = D_{\left(n_1,\dots,n_r\right)} \ltimes N_{\left(n_1,\dots,n_r\right)}$.

Given a composition $\left(n_1,\dots,n_r\right)$ of $n$ and representations $\sigma_1, \dots, \sigma_r$ of $\GL_{n_1}\left(\finiteField\right)$, $\dots$, $\GL_{n_r}\left(\finiteField\right)$, respectively, we define the inflation $\sigma_1 \bar{\otimes} \dots \bar{\otimes} \sigma_r$ as the representation of $P_{\left(n_1,\dots,n_r\right)}$, acting on the space of $\sigma_1 \otimes \dots \otimes \sigma_r$ by $$\left(\sigma_1 \bar{\otimes} \dots \bar{\otimes} \sigma_r\right)\left(du\right) = \sigma_1\left(g_1\right) \otimes \dots \otimes \sigma_r\left(g_r\right),$$
where $d = \diag\left(g_1,\dots,g_r\right) \in D_{\left(n_1,\dots,n_r\right)}$ and $u \in N_{\left(n_1,\dots,n_r\right)}$. The \emph{parabolic induction} of $\sigma_1, \dots, \sigma_r$, denoted by $\sigma_1 \circ \dots \circ \sigma_r$, is a representation of $\GL_n\left(\finiteField\right)$, defined as the following induced representation $$\sigma_1 \circ \dots \circ \sigma_r = \Ind{P_{\left(n_1,\dots,n_r\right)}}{\GL_n\left(\finiteField\right)}{\sigma_1 \bar{\otimes} \dots \bar{\otimes} \sigma_r}.$$
Parabolic induction is commutative and associative. By this we mean that $\sigma_1 \circ \sigma_2$ is isomorphic to $\sigma_2 \circ \sigma_1$, and that both $\left(\sigma_1 \circ \sigma_2\right) \circ \sigma_3$ and $\sigma_1 \circ \left(\sigma_2 \circ \sigma_3\right)$ are isomorphic to $\sigma_1 \circ \sigma_2 \circ \sigma_3$, for any finite dimensional representations $\sigma_1$, $\sigma_2$ and $\sigma_3$ of $\GL_{n_1}\left(\finiteField\right)$, $\GL_{n_2}\left(\finiteField\right)$ and $\GL_{n_3}\left(\finiteField\right)$, respectively. See \cite[Page 17, Section 0.5]{Gelfand70} or \cite{Green55}.

A representation $\pi$ of $\GL_n\left(\finiteField\right)$ is called \emph{cuspidal} if for every composition $\left(n_1,\dots,n_r\right) \ne \left(n\right)$ of $n$, $\pi$ does not admit non-zero $N_{\left(n_1,\dots,n_r\right)}$-invariant vectors, that is, if $v \in \pi$ is an element such that $\pi\left(u\right) v = v$ for every $u \in N_{\left(n_1,\dots,n_r\right)}$, then $v = 0$. If $\pi$ is irreducible, $\pi$ is cuspidal if and only if $\pi$ is not a subrepresentation of $\sigma_1 \circ \dots \circ \sigma_r$ for any $\sigma_1, \dots, \sigma_r$ as above, where $r > 1$.

A character $\theta \colon \multiplicativegroup{\finiteFieldExtension{n}} \to \multiplicativegroup{\cComplex}$ is called \emph{regular} if $\theta^{q^j} \ne \theta$ for every $1 \le j \le n-1$. Equivalently, $\theta$ is regular if its Frobenius orbit $f\left(\theta\right) = \{\theta^{q^j} \mid j \ge 0\}$ is of size $n$. By \cite[Theorem 6.1]{Gelfand70}, irreducible cuspidal representations are in bijection with Frobenius orbits of regular characters $\theta \colon \multiplicativegroup{\finiteFieldExtension{n}} \to \multiplicativegroup{\cComplex}$.

By \cite[Theorem 2.4]{Gelfand70}, for every irreducible representation $\pi$ of $\GL_n\left(\finiteField\right)$, there exist a composition $\left(n_1, \dots, n_r\right)$ of $n$, and irreducible cuspidal representations $\sigma_1, \dots, \sigma_r$ of $\GL_{n_1}\left(\finiteField\right), \dots, \GL_{n_r}\left(\finiteField\right)$, respectively, such that $\pi$ is an irreducible subrepresentation of the parabolic induction $\sigma_1 \circ \dots \circ \sigma_r$. The multiset $\left\{\sigma_1, \dots, \sigma_r\right\}$ is completely determined by $\pi$, and we call it \emph{the cuspidal support of $\pi$}.

We move to describe the parameterization of irreducible representations of $\GL_n\left(\finiteField\right)$. Let $\Irr\left(\GL_n\left(\finiteField\right)\right)$ be the set of all irreducible representations of $\GL_n\left(\finiteField\right)$ (up to isomorphism), and let $\IrrCusp\left(\GL_n\left(\finiteField\right)\right) \subset \Irr\left(\GL_n\left(\finiteField\right)\right)$ denote the subset of irreducible cuspidal representations of $\GL_n\left(\finiteField\right)$. A \emph{parameter} $\varphi$ is an assignment $\left(d, \sigma\right) \mapsto \varphi\left(d, \sigma\right)$, where $d$ is a positive integer, $\sigma \in \IrrCusp\left(\GL_d\left(\finiteField\right)\right)$ and $\varphi\left(d, \sigma\right)$ is a partition, where for almost every $d$ and $\sigma$, $\varphi\left(d, \sigma\right) = \left(\right)$. We say that $\varphi$ is supported on $\left(d,\sigma\right)$ if $\varphi\left(d, \sigma\right) \ne \left(\right)$. For a parameter $\varphi$, we define $$\abs{\varphi} = \sum_{\left( d, \sigma\right)} d \cdot \sizeof{\varphi\left(d, \varphi\right)}.$$
Then by \cite{Green55} or \cite[Section 1]{Macdonald80}, the irreducible representations of $\GL_n\left(\finiteField\right)$ are in bijection with parameters $\varphi$, such that $\abs{\varphi} = n$. Let us discuss this bijection very briefly. Let $\sigma$ be an irreducible cuspidal representation of $\GL_d\left(\finiteField\right)$. For any partition $\lambda \vdash m$, one can define an irreducible representation $\sigma^{\lambda}$ of $\GL_{dm}\left(\finiteField\right)$ by the following steps: define $\sigma^{\left(1\right)} = \sigma$. If $m > 1$, consider the parabolically induced representation $\sigma^{\circ m} = \sigma \circ \dots \circ \sigma$. Then $\sigma^{\circ m}$ is not irreducible, and its irreducible components correspond to irreducible representations of the symmetric group $\SymmetricGroup_m$, which in turn correspond to partitions $\lambda \vdash m$. We denote by $\sigma^{\lambda}$ the irreducible representation corresponding to the partition $\lambda$. See also \cite[Appendix B]{gurevich2021harmonic}. Finally, given a parameter $\varphi$ with $\abs{\varphi} = n$, let $\left(d_j, \sigma_j \right)_{j=1}^s$ be all the elements with $\varphi\left(d_j, \sigma_j\right) \ne \left(\right)$. Then we define a representation $\pi\left(\varphi\right)$ by the parabolic induction
$$\pi\left(\varphi\right) = \sigma_1^{\varphi\left(d_1, \sigma_1\right)} \circ \dots \circ \sigma_s^{\varphi\left(d_s, \sigma_s\right)}.$$
In particular, the cuspidal support of $\pi\left(\varphi\right)$ is the multiset containing exactly $\sizeof{\varphi\left(d_j, \sigma_j\right)}$ copies of $\sigma_j$, for every $1 \le j \le s$.

We conclude by remarking that by choosing a sequence $j = \left(j_n\right)_{n=1}^{\infty}$ of isomorphisms $j_n \colon \multiplicativegroup{\finiteFieldExtension{n}} \to \charactergroup{n}$ for every $n \ge 1$, such that for every $d \mid n$, we have ${j_n}\left(x\right) = j_d\left(x\right) \circ \FieldNorm{n}{d}$, for every $x \in \multiplicativegroup{\finiteFieldExtension{d}}$, we may establish a bijection between conjugacy classes of $\GL_n\left(\finiteField\right)$ and irreducible representations of $\GL_n\left(\finiteField\right)$. Given a cuspidal representation $\sigma$ of $\GL_d\left(\finiteField\right)$ corresponding to the Frobenius orbit $\{ \theta, \theta^q, \dots \theta^{q^{d-1}} \}$, where $\theta \colon \multiplicativegroup{\finiteFieldExtension{d}} \to \multiplicativegroup{\cComplex}$ is a regular character, we define $C_j\left(\sigma\right)$ to be a regular elliptic element of $\GL_d\left(\finiteField\right)$ with eigenvalues $\{ j_d^{-1}\left(\theta\right), j_d^{-1}\left(\theta\right)^q, \dots, j_d^{-1}\left(\theta\right)^{q^{d-1}} \}$ over $\algebraicClosure{\finiteField}$. Given a parameter $\varphi$ with $\sizeof{\varphi} = n$, supported on $\left(d_i, \sigma_i\right)_{i=1}^s$, we define $$C_j\left(\pi\left(\varphi\right)\right) = \diag\left(J_{\varphi\left(d_1, \sigma_1\right)}\left(C_j\left(\sigma_1\right)\right), \dots, J_{\varphi\left(d_s, \sigma_s\right)}\left(C_j\left(\sigma_d\right)\right)\right).$$
The map sending an irreducible representation $\pi$ to the conjugacy class of $C_j\left(\pi\right)$ is a bijection as desired.

\subsection{Non-abelian Gauss sums}\label{subsec:non-abelian-gauss-sums}

Let $\pi$ be an irreducible representation of $\GL_n\left(\finiteField\right)$, and let $\chi \colon \multiplicativegroup{\finiteField} \to \multiplicativegroup{\cComplex}$ be a character. Let $\fieldCharacter \colon \finiteField \to \multiplicativegroup{\cComplex}$ be a non-trivial character. Consider the assignment $g \mapsto \chi\left(\det g\right) \fieldCharacter\left(\trace g\right)$. This assignment is a class function of $\GL_n\left(\finiteField\right)$. Therefore, we may define an operator $\GKGaussSum{\pi}{\chi}{\fieldCharacter} \in \Hom_{\GL_n\left(\finiteField\right)}\left(\pi, \pi\right)$ by the formula $$\GKGaussSum{\pi}{\chi}{\fieldCharacter} = q^{-\frac{n^2}{2}} \sum_{g \in \GL_n\left(\finiteField\right)} \fieldCharacter\left(\trace g\right) \chi\left(\det g\right) \pi\left(g\right).$$
Since $\pi$ is irreducible, by Schur's lemma there exists a constant $\GKGaussSumScalar{\pi}{\chi}{\fieldCharacter} \in \cComplex$, such that $\GKGaussSum{\pi}{\chi}{\fieldCharacter} = \GKGaussSumScalar{\pi}{\chi}{\fieldCharacter} \cdot \idmap_{\pi}$. This scalar was explicitly computed by Kondo \cite{Kondo1963}, see also \cite[Section 2]{Macdonald80} and \cite[Page 288]{macdonald1998symmetric}. We state his result in the following theorem.

\begin{theorem}\label{thm:kondos-formula}
	Let $\pi$ be an irreducible representation of $\GL_n\left(\finiteField\right)$. Suppose that the cuspidal support of $\pi$ is $\left\{\sigma_1,\dots,\sigma_r\right\}$, where for every $1 \le j \le r$, $\sigma_j$ is an irreducible cuspidal representation of $\GL_{n_j}\left(\finiteField\right)$. Suppose that $\sigma_j$ corresponds to the Frobenius orbit $\{\theta_j, \theta_j^q, \dots, \theta_j^{q^{n_j - 1}}\}$, where $\theta_j \colon \multiplicativegroup{\finiteFieldExtension{n_j}} \to \multiplicativegroup{\cComplex}$ is a regular character. Then
	$$ \GKGaussSumScalar{\pi}{\chi}{\fieldCharacter} = \left(-1\right)^n q^{-\frac{n}{2}} \prod_{j=1}^r \GaussSumCharacter{n_j}{\theta_j}{\chi}{\fieldCharacter},$$
	where $\GaussSumCharacter{n_j}{\theta_j}{\chi}{\fieldCharacter}$ is the Gauss sum $$\GaussSumCharacter{n_j}{\theta_j}{\chi}{\fieldCharacter} = -\sum_{t \in \multiplicativegroup{\finiteFieldExtension{n_j}}} \theta_j\left(t\right) \chi\left(\FieldNorm{n_j}{1}\left(t\right)\right) \fieldCharacter_{n_j}\left(t\right).$$
\end{theorem}

\section{Formula for regular elliptic elements}

This section is dedicated to the proof of our first main theorem, providing a formula for regular elliptic elements. We first write an expansion for $\ExoticKloosterman\left(\alpha, \fieldCharacter, x\right)$ in terms of non-abelian Gauss sums and characters of irreducible representations of $\GL_n\left(\finiteField\right)$. This is done in \Cref{prop:kloosterman-character-identity}. Then, the idea is to substitute a regular elliptic element $x$ in the expansion we obtained. For this, we need to know for which irreducible representations $\pi$ of $\GL_n\left(\finiteField\right)$ the value $\trace \pi\left(x\right)$ is non-zero, and how to compute it in this case. We recall these results in \Cref{subsec:representations-supported-on-regular-elliptic}. Finally, in \Cref{subsec:proof-of-theorem-for-regular-elliptic-elements} we prove \Cref{thm:regular-elliptic-elements}, our theorem regarding regular elliptic elements.

\subsection{A character identity for matrix Kloosterman sums}\label{subsec:character-identity-for-matrix-kloosterman-sums}

In this section, we write an identity that expresses twisted matrix Kloosterman sums in terms of the non-abelian Gauss sums discussed in \Cref{subsec:non-abelian-gauss-sums} and irreducible characters of $\GL_n\left(\finiteField\right)$.

\begin{proposition}\label{prop:kloosterman-character-identity}
	Let $\alpha \colon \left(\multiplicativegroup{\finiteField}\right)^k \to \multiplicativegroup{\cComplex}$ be a character. Write $\alpha = \alpha_1 \times \dots \times \alpha_k$, where $\alpha_1, \dots, \alpha_k \colon \multiplicativegroup{\finiteField} \to \multiplicativegroup{\cComplex}$ are characters. Then for any $x \in \GL_n\left(\finiteField\right)$,
	$$ \ExoticKloosterman\left(\alpha,\fieldCharacter,x\right) = \frac{q^{\frac{kn^2}{2}}}{\sizeof{\GL_n\left(\finiteField\right)}} \sum_{\pi \in \Irr\left(\GL_n\left(\finiteField\right)\right)} \left(\prod_{i=1}^k \GKGaussSumScalar{\pi}{\alpha_i}{\fieldCharacter}\right) \dim \pi \cdot \conjugate{\trace \pi\left(x\right)}.$$
\end{proposition}
\begin{proof}
	Let $\pi$ be an irreducible representation of $\GL_n\left(\finiteField\right)$. By applying the operators $\GKGaussSum{\pi}{\alpha_1}{\fieldCharacter}$, $\dots$, $\GKGaussSum{\pi}{\alpha_k}{\fieldCharacter}$ repeatedly, we have the identity
	$$\left(\prod_{i=1}^k \GKGaussSumScalar{\pi}{\alpha_i}{\fieldCharacter} \right) \cdot \idmap_{\pi} = q^{-\frac{kn^2}{2}} \sum_{g_1, \dots, g_k \in \GL_n\left(\finiteField\right)} \fieldCharacter\left(\trace \sum_{i=1}^k g_i\right) \left(\prod_{i=1}^k\alpha_i\left(\det g_i\right)\right) \pi\left( g_1 \cdot \dots \cdot g_k \right).$$
	Taking the trace of both sides, we get
	\begin{equation}\label{eq:trace-of-product-of-non-abelian-gauss-sum}
		\begin{split}
			& q^{-\frac{kn^2}{2}} \sum_{g_1, \dots, g_k \in \GL_n\left(\finiteField\right)} \fieldCharacter\left(\trace \sum_{i=1}^k g_j\right) \left(\prod_{i=1}^k\alpha_i\left(\det g_i\right)\right) \trace \pi\left( g_1 \cdot \dots \cdot g_k \right) \\
			=& \left(\prod_{i=1}^k \GKGaussSumScalar{\pi}{\alpha_i}{\fieldCharacter} \right) \cdot \dim \pi.
		\end{split}
	\end{equation}
	Multiplying both sides by $\conjugate{\trace \pi\left(x\right)}$ and summing over $\pi \in \Irr\left(\GL_n\left(\finiteField\right)\right)$, we get, by Schur's orthogonality relations
	\begin{align*}
		& \sizeof{C_{\GL_n\left(\finiteField\right)}\left(x\right)} q^{-\frac{kn^2}{2}} \sum_{\substack{g_1, \dots, g_k \in \GL_n\left(\finiteField\right)\\
				g_1 \cdot \dots \cdot g_k \sim x}} \fieldCharacter\left(\trace \sum_{i=1}^k g_j\right) \left(\prod_{i=1}^k\alpha_i\left(\det g_j\right)\right) \\
			=& \sum_{\pi \in \Irr\left(\GL_n\left(\finiteField\right)\right)} \left(\prod_{i=1}^k \GKGaussSumScalar{\pi}{\alpha_i}{\fieldCharacter} \right) \cdot \dim \pi \cdot \conjugate{\trace {\pi\left(x\right)}},
	\end{align*}
	where $C_{\GL_n\left(\finiteField\right)}\left(x\right)$ is the the centralizer of $x$ in $\GL_n\left(\finiteField\right)$, and for $h_1, h_2 \in \GL_n\left(\finiteField\right)$, we write $h_1 \sim h_2$ if $h_1$ and $h_2$ are conjugate.
	Hence, we have \begin{equation}\label{eq:kloosterman-identity-before-orbit-stabilizer-lemma}
		\begin{split}
			\sizeof{C_{\GL_n\left(\finiteField\right)}\left(x\right)} q^{-\frac{kn^2}{2}} \sum_{y \in \left[x\right]} \ExoticKloosterman\left(\alpha, \fieldCharacter, y\right) &= \sum_{\pi \in \Irr\left(\GL_n\left(\finiteField\right)\right)} \left(\prod_{i=1}^k \GKGaussSumScalar{\pi}{\alpha_i}{\fieldCharacter} \right) \cdot \dim \pi \cdot \conjugate{\trace {\pi\left(x\right)}},
		\end{split}
	\end{equation}
	where $\left[x\right]$ is the conjugacy class of $x$ in $\GL_n\left(\finiteField\right)$. Since $\ExoticKloosterman\left(\alpha, \fieldCharacter, y\right) = \ExoticKloosterman\left(\alpha, \fieldCharacter, x\right)$ for any $y$ conjugate to $x$, we have that the left hand side of \eqref{eq:kloosterman-identity-before-orbit-stabilizer-lemma} is $\sizeof{C_{\GL_n\left(\finiteField\right)}\left(x\right)} \cdot \sizeof{\left[x\right]} \cdot \ExoticKloosterman\left(\alpha, \fieldCharacter, x\right)$, which by the orbit-stabilizer theorem is $\sizeof{\GL_n\left(\finiteField\right)} \cdot \ExoticKloosterman\left(\alpha, \fieldCharacter, x\right)$. The result now follows.
\end{proof}

\subsection{Representations supported on regular elliptic elements}\label{subsec:representations-supported-on-regular-elliptic}
Our idea is to use the identity from \Cref{subsec:character-identity-for-matrix-kloosterman-sums} for a regular elliptic element $x \in \GL_n\left(\finiteField\right)$. In order to do that, we need to know for which $\pi$, the value $\trace \pi\left(x\right)$ does not vanish, and how to compute this value. These are questions that can be answered by Green's formulas for the irreducible characters of $\GL_n\left(\finiteField\right)$ \cite{Green55}. Fortunately, these questions have been answered completely by Graham Gordon in his doctoral thesis \cite{Gordon2022}.

\begin{theorem}\label{thm:support-of-elliptic-elements}
	Let $x \in \GL_n\left(\finiteField\right)$ be a regular elliptic element. Suppose that the roots of the characteristic polynomial of $x$ are $\xi, \xi^q, \dots, \xi^{q^{n-1}}$. Let $\pi$ be an irreducible representation of $\GL_n\left(\finiteField\right)$, and suppose that $\trace \pi\left(x\right) \ne 0$. Then
	\begin{enumerate}
		\item (\cite[Corollary 3.7]{Gordon2022}) There exist integers $a,b,r$, such that $ab = n$, $0 \le r \le b-1$, and an irreducible cuspidal representation $\sigma$ of $\GL_a\left(\finiteField\right)$, such that $\pi$ is the irreducible representation of $\GL_n\left(\finiteField\right)$ corresponding to the parameter supported on $\left(a, \sigma\right)$ with value $\left(b-r, 1^r\right)$ at $\left(a, \sigma\right)$, i.e., $\pi = \sigma^{\left(b-r, 1^{r}\right)}$.
		\item (\cite[Corollary 3.7]{Gordon2022}) $$\dim \sigma^{\left(b-r, 1^{r}\right)} = q^{a \binom{r+1}{2}} \cdot \frac{\prod_{j=1}^{ab}\left(q^j - 1\right)}{\prod_{j=1}^b \left(q^{aj}-1\right)} \cdot \binom{b-1}{r}_{q^a},$$
		where $\binom{n}{k}_q$ the $q$-binomial.
		\item (\cite[Theorem 1.7]{Gordon2022} and \cite[Corollary 2.3]{Gordon2022}) Suppose that $\sigma$ corresponds to the Frobenius orbit $\{\theta, \theta^q, \dots, \theta^{q^{a-1}} \}$, where $\theta \colon \multiplicativegroup{\finiteFieldExtension{a}} \to \multiplicativegroup{\cComplex}$ is a regular character. Then $$\trace \sigma^{\left(b-r, 1^{r}\right)} \left(x\right) = \left(-1\right)^{b\left(a-1\right)} \left(-1\right)^r \frac{1}{b} \sum_{j=0}^{ab - 1} \theta\left(\FieldNorm{ab}{a}\left(\xi^{q^j}\right)\right).$$
	\end{enumerate}
\end{theorem}

\subsection{Proof of the theorem for regular elliptic elements}\label{subsec:proof-of-theorem-for-regular-elliptic-elements}
We are ready to prove our first main theorem.
\begin{theorem}\label{thm:regular-elliptic-elements}
	For any character $\alpha \colon \left(\multiplicativegroup{\finiteField}\right)^k \to \multiplicativegroup{\cComplex}$ and any regular elliptic element $\xi \in \GL_n\left(\finiteField\right)$ with eigenvalues $\{\xi,\xi^q,\dots,\xi^{q^{n-1}}\}$, where $\xi \in \multiplicativegroup{\finiteFieldExtension{n}}$, we have $$ \ExoticKloosterman\left(\alpha, \fieldCharacter, x\right) = \left(-1\right)^{\left(n + 1\right)\left(k - 1\right)} q^{\left(k-1\right)\binom{n}{2}} \ExoticKloosterman_n\left(\alpha, \fieldCharacter, \xi\right).$$
\end{theorem}
\begin{proof}
	We will use \Cref{prop:kloosterman-character-identity}. By \Cref{thm:support-of-elliptic-elements}, we may reduce the summation to \begin{equation}\label{eq:reduced-kloosterman-identity}
		\begin{split}
				\ExoticKloosterman\left(\alpha,\fieldCharacter,x\right) =& \frac{q^{\frac{kn^2}{2}}}{\sizeof{\GL_n\left(\finiteField\right)}} \sum_{\substack{a,b\\
					ab = n}} \sum_{\sigma \in \IrrCusp\left(\GL_a\left(\finiteField\right)\right)} \sum_{r = 0}^{b-1} \left(\prod_{i=1}^k \GKGaussSumScalar{\sigma^{\left(b-r,1^r\right)}}{\alpha_i}{\fieldCharacter}\right)\\
			& \times \dim \sigma^{\left(b-r,1^r\right)} \cdot \conjugate{{\trace \sigma^{\left(b-r,1^r\right)}\left(x\right)}}.
		\end{split}
	\end{equation}
	By \Cref{thm:kondos-formula}, we have that for every $0 \le r \le b-1$, \begin{equation}
		\GKGaussSumScalar{\sigma^{\left(b-r,1^r\right)}}{\alpha_i}{\fieldCharacter} = \GKGaussSumScalar{\sigma^{\left(b\right)}}{\alpha_i}{\fieldCharacter}.
	\end{equation}
	We start by evaluating the sum over $r$ in \eqref{eq:reduced-kloosterman-identity}. For any $a \mid n$ and any irreducible cuspidal representation $\sigma$ of $\GL_a\left(\finiteField\right)$, choose a regular character $\theta_{\sigma} \colon \multiplicativegroup{\finiteFieldExtension{a}} \to \multiplicativegroup{\cComplex}$ representing the Frobenius orbit associated to $\sigma$.
	By \Cref{thm:support-of-elliptic-elements}, we have
	\begin{equation}\label{eq:explicit-values-for-elliptic-support}
		\begin{split}
			\sum_{r=0}^{b-1} \dim \sigma^{\left(b-r,1^r\right)} \cdot \conjugate{\trace \sigma^{\left(b-r,1^r\right)}\left(x\right)} = &  \left(-1\right)^{b\left(a-1\right)} \frac{1}{b} \sum_{j=0}^{ab - 1} \theta_{\sigma} \left(\FieldNorm{ab}{a}\left(\xi^{-q^j}\right)\right) \frac{\prod_{j=1}^{ab}\left(q^j - 1\right)}{\prod_{j=1}^b \left(q^{aj}-1\right)} \\
			& \times \sum_{r=0}^{b-1} \left(-1\right)^r q^{a \binom{r+1}{2}} \binom{b-1}{r}_{q^a}.
		\end{split}
	\end{equation}
	Using the $q$-binomial theorem, we have
	\begin{equation}\label{eq:q-binomial-theorem}
		\sum_{r=0}^{b-1} \left(-1\right)^r q^{a \binom{r+1}{2}} \binom{b-1}{r}_{q^a} = \prod_{j=0}^{b-2} \left(1 - q^{a\left(j+1\right)}\right) = \left(-1\right)^{b-1} \prod_{j=1}^{b-1} \left(q^{aj} - 1\right).
	\end{equation}
	Notice that for every $\eta \in \multiplicativegroup{\finiteFieldExtension{ab}}$, we have $\FieldNorm{ab}{a}\left(\eta\right)^{q^a} = \FieldNorm{ab}{a}\left(\eta\right)$. Therefore, we have \begin{equation}\label{eq:reduced-frobenius-orbit}
		\frac{1}{b} \sum_{j=0}^{ab - 1} \theta_{\sigma}\left(\xi^{-q^j}\right) = \sum_{j=0}^{a-1} \theta_{\sigma}^{q^j}\left(\xi^{-1}\right).
	\end{equation}
	Hence, combining \eqref{eq:reduced-kloosterman-identity}-\eqref{eq:reduced-frobenius-orbit} with the facts $\sizeof{\GL_n\left(\finiteField\right)} = q^{\binom{n}{2}} \prod_{j=1}^{n}\left(q^j - 1\right)$ and $\sizeof{\charactergroup{n}} = q^n - 1$, we get that $\ExoticKloosterman\left(\alpha, \fieldCharacter, x\right)$ is given by
	$$\left(-1\right)^{n-1} \frac{q^{\frac{kn^2}{2} - \binom{n}{2}}}{\sizeof{\charactergroup{n}}} \sum_{\substack{a,b\\
			ab = n}} \sum_{\sigma \in \IrrCusp\left(\GL_a\left(\finiteField\right)\right) } \left(\prod_{i=1}^k \GKGaussSumScalar{\sigma^{\left(b\right)}}{\alpha_i}{\fieldCharacter}\right) \sum_{j=0}^{a - 1} \theta_{\sigma}^{q^j}\left(\FieldNorm{ab}{a}\left(\xi^{-1}\right)\right).$$
		We would like to replace the triple sum with a single sum over $\beta \in \charactergroup{n}$. To do so, notice that by \Cref{thm:kondos-formula}, we have
		$$\GKGaussSumScalar{\sigma^{\left(b\right)}}{\alpha_i}{\fieldCharacter} = \left(-1\right)^n \GaussSumCharacter{a}{\theta_{\sigma}}{\alpha_i}{\fieldCharacter}^b$$
		and by the Hasse--Davenport relation, we have
		$$\GKGaussSumScalar{\sigma^{\left(b\right)}}{\alpha_i}{\fieldCharacter} = \left(-1\right)^n q^{-\frac{ab}{2}} \GaussSumCharacter{a}{\theta_{\sigma}}{\alpha_i}{\fieldCharacter}^b = \left(-1\right)^n q^{-\frac{ab}{2}} \GaussSumCharacter{ab}{\theta_{\sigma} \circ \FieldNorm{ab}{a}}{\alpha_i}{\fieldCharacter}.$$
		Hence, $\ExoticKloosterman\left(\alpha, \fieldCharacter, x\right)$ is given by
		\begin{equation}\label{eq:uniform-sum-for-matrix-kloosterman-sum}
			\left(-1\right)^{n\left(k-1\right)-1} \frac{q^{\left(k-1\right)\binom{n}{2}}}{\sizeof{\charactergroup{n}}} \sum_{\substack{a,b\\
					ab = n}} \sum_{\sigma \in \IrrCusp\left(\GL_a\left(\finiteField\right)\right) } \sum_{j=0}^{a - 1} \left(\theta_{\sigma} \circ \FieldNorm{n}{a}\right)^{q^j} \left(\xi^{-1}\right) \left(\prod_{i=1}^k \GaussSumCharacter{n}{\left(\theta_{\sigma} \circ \FieldNorm{n}{a}\right)^{q^j}}{\alpha_i}{\fieldCharacter}\right),
		\end{equation}
	where we used the fact that Gauss sums are constant on Frobenius orbits.
	Every character $\beta \colon \multiplicativegroup{\finiteFieldExtension{n}} \to \multiplicativegroup{\cComplex}$ appears exactly once as $\theta_{\sigma} \circ \FieldNorm{n}{a}^{q^j}$ in \eqref{eq:uniform-sum-for-matrix-kloosterman-sum}. Indeed, given such $\beta$, let $a$ be the cardinality of the set $\left\{\beta^{q^j} \mid j \ge 0\right\}$. Then $a \mid n$, and there exists a regular character $\beta' \colon \multiplicativegroup{\finiteFieldExtension{a}} \to \multiplicativegroup{\cComplex}$, such that $\beta = \beta' \circ \FieldNorm{n}{a}$. We take $\sigma$ to be the cuspidal representation of $\GL_a\left(\finiteField\right)$ corresponding to the Frobenius orbit of $\beta'$. Then $\beta' = \theta_{\sigma}^{q^j}$ for some $j$. This shows that every character appears in \eqref{eq:uniform-sum-for-matrix-kloosterman-sum}. It appears exactly once, because Frobenius orbits of regular characters of $\multiplicativegroup{\finiteFieldExtension{a}}$ are in bijection with irreducible cuspidal representations of $\GL_a\left(\finiteField\right)$. Hence, we can rewrite \eqref{eq:uniform-sum-for-matrix-kloosterman-sum} and obtain the following expression for $\ExoticKloosterman\left(\alpha, \fieldCharacter, x\right)$:
	$$\left(-1\right)^{\left(k-1\right)\left(n+1\right)} \frac{q^{\left(k-1\right)\binom{n}{2}}}{\sizeof{\charactergroup{n}}} \sum_{\beta \in \charactergroup{n}} \sum_{t_1,\dots,t_k \in \multiplicativegroup{\finiteFieldExtension{n}}} \beta\left( \xi^{-1} \cdot \prod_{i=1}^k t_i \right) \left(\prod_{i=1}^k \alpha_i\left(  \FieldNorm{n}{1}\left( t_i \right)\right)  \right) \fieldCharacter_{n_j}\left(\sum_{i=1}^k t_i\right).$$
	For any $b \in \multiplicativegroup{\finiteFieldExtension{n}}$, the assignment $\charactergroup{n} \to \multiplicativegroup{\cComplex}$ given by $\beta \mapsto \beta\left(b\right)$ is a character. Since an average of a character over a finite group results with $0$, unless the character is trivial, we get by changing the order of the summation that the sum over $\beta$ is zero unless $\xi^{-1} \cdot \prod_{i=1}^k t_i = 1$, in which case the average evaluates to $1$. Therefore, we get
	$$\ExoticKloosterman\left(\alpha, \fieldCharacter, x\right) =  \left(-1\right)^{\left(n+1\right)\left(k-1\right)} q^{\left(k-1\right)\binom{n}{2}} \sum_{\substack{t_1,\dots,t_k \in \multiplicativegroup{\finiteFieldExtension{n}}\\
	\prod_{i=1}^k t_i = \xi}} \left(\prod_{i=1}^k \alpha_i\left(  \FieldNorm{n}{1}\left( t_i \right)  \right)\right) \fieldCharacter_{n_j}\left(\sum_{i=1}^k t_i\right),$$
	as required.
\end{proof}

We finish this section by remarking that by using the multiplicativity property of twisted matrix Kloosterman sums (\Cref{thm:multiplicativity-property}), we immediately get a formula for regular semisimple elements, as follows.
\begin{corollary}
	If $y_1 \in \GL_{d_1}\left(\finiteField\right), \dots, y_s \in \GL_{d_s}\left(\finiteField\right)$ are regular elliptic elements with mutually disjoint eigenvalues, then
	$$ \ExoticKloosterman\left(\alpha, \fieldCharacter, \diag\left(y_1, \dots, y_s\right)\right) = \left(-1\right)^{\left(n+s \right)\left(k-1\right)} q^{\left(k-1\right)\binom{n}{2} } \prod_{j = 1}^s \ExoticKloosterman_{d_j}\left( \alpha, \fieldCharacter, \eta_j \right),$$
	where $n = \sum_{j=1}^s d_j$ and $\eta_j \in \multiplicativegroup{\finiteFieldExtension{d_j}}$ is such that $\{\eta_j, \eta_j^q, \dots, \eta_j^{q^{d_j-1}}\}$ are the eigenvalues of $y_j$.
\end{corollary}

\section{Formula for general elements}

This section is dedicated to the proof of our second main theorem, providing a formula for twisted matrix Kloosterman sums of general Jordan matrices $J_{\mu}\left(x\right)$. The idea is to use a different basis for the space of class functions of $\GL_n\left(\finiteField\right)$, the basis of Green's basic characters. We first recall the definition of Green's basic characters. We use them to obtain an expression for $\ExoticKloosterman\left(\alpha, \fieldCharacter, J_{\mu}\left(x\right)\right)$ in terms of Green polynomials and twisted Kloosterman sums. We then recall the definition of modified Hall--Littlewood polynomials and twisted Kloosterman sheaves, and use these in order to reformulate our expression for $\ExoticKloosterman\left(\alpha, \fieldCharacter, J_{\mu}\left(x\right)\right)$.

\subsection{Green's basic characters}\label{subsec:green-basic-characters}
In this section, we recall the definition of Green's basic characters. They form another basis for the space of class functions of $\GL_n\left(\finiteField\right)$, and hence can be used to give an alternative expression for the twisted matrix Kloosterman sums of our interest.

A \emph{class function} $F \colon \GL_n\left(\finiteField\right) \to \cComplex$ is a function that is constant on conjugacy classes of $\GL_n\left(\finiteField\right)$. The space of such functions is equipped with the following inner product $$\innerproduct{F_1}{F_2} = \frac{1}{\sizeof{\GL_n\left(\finiteField\right)}} \sum_{x \in \GL_n\left(\finiteField\right)} F_1\left(x\right) \conjugate{F_2\left(x\right)}.$$

For any character $\theta \colon \multiplicativegroup{\finiteFieldExtension{n}} \to \multiplicativegroup{\cComplex}$, not necessarily regular, let $f\left({\theta}\right) = \{ \theta^{q^j} \mid j \ge 0 \}$ be its Frobenius orbit. Denote the cardinality of $f\left({\theta}\right)$ by $\deg f\left({\theta}\right)$ or $\deg \theta$. We may define a class function $\chi_{f\left(\theta\right)} \colon \GL_n\left(\finiteField\right) \to \cComplex$ associated to $f\left({\theta}\right)$ as follows. $\chi_{f\left(\theta\right)}$ is only supported on conjugacy classes of the form $J_{\mu}\left(x\right)$ where  $\mu$ is a partition of $b$, where $ab = n$ and $x \in \GL_a\left(\finiteField\right)$ is a regular elliptic element. Suppose that such $x$ has eigenvalues $\{\xi, \xi^q, \dots, \xi^{q^{a-1}}\}$, where $\xi \in \multiplicativegroup{\finiteFieldExtension{a}}$, then we set
$$\chi_{f\left(\theta\right)}\left(J_{\mu}\left(x\right)\right) = \Phi_{\ell\left(\mu\right)}\left(q^a\right) \cdot \sum_{j = 0}^{a-1} \theta\left(\xi^{q^j}\right),$$ where
$$\Phi_{l}\left(T\right) = \prod_{j=1}^{l-1} \left(1 - T^j\right).$$
If $\theta$ is a regular character, then $\chi_{f\left(\theta\right)}$ equals $\left(-1\right)^{n - 1} \trace \sigma$, where $\sigma$ is the irreducible cuspidal representation of $\GL_n\left(\finiteField\right)$ corresponding to the Frobenius orbit $\{\theta, \theta^q, \dots, \theta^{q^{n-1}}\}$. Otherwise, $\pm \chi_{f\left(\theta\right)}$ is not a character of an irreducible representation of $\GL_n\left(\finiteField\right)$.

If $n = n_1 + \dots + n_r$ and $\chi_1, \dots, \chi_r$ are class functions of $\GL_{n_1}\left(\finiteField\right)$, $\dots$, $\GL_{n_r}\left(\finiteField\right)$, respectively, we define a class function $\chi_1 \circ \dots \circ \chi_r$ of $\GL_n\left(\finiteField\right)$ by the formula for parabolic induction
$$ \left(\chi_1 \circ \dots \circ \chi_r\right)\left(g\right) = \frac{1}{\sizeof{P_{\left(n_1, \dots, n_r\right)}}} \sum_{\substack{h \in \GL_n\left(\finiteField\right)\\
hgh^{-1} \in \diag\left(g_1,\dots,g_r\right) \cdot N_{\left(n_1,\dots,n_r\right)}}} \chi_1\left(g_1\right) \cdot \dots \cdot \chi_r\left(g_r\right),$$
where $g_1 \in \GL_{n_1}\left(\finiteField\right)$, $\dots$, $g_r \in \GL_{n_r}\left(\finiteField\right)$. The operation $\circ$ is commutative and associative.

Given a partition $\lambda = \left(n_1,\dots,n_r\right) \vdash n$ and Frobenius orbits $f_1 = f\left(\theta_1\right)$, $\dots$, $f_r = f\left(\theta_r\right)$, where $\theta_j \colon \multiplicativegroup{\finiteFieldExtension{n_j}} \to \multiplicativegroup{\cComplex}$, we define a class function of $\GL_n\left(\finiteField\right)$ associated to the multiset $\left\{f_1, \dots, f_r\right\}$ by
$$\chi^{\lambda}_{\left\{f_1, \dots, f_r\right\}} = \chi_{f_1} \circ \dots \circ \chi_{f_r}.$$
This is well defined, as $\circ$ is commutative and associative.

A function of the form $\chi^{\lambda}_{\left\{f_1, \dots, f_r\right\}}$ is called a \emph{Green basic character}. Green showed in \cite[Lemma 6.3]{Green55} that the collection of all basic characters corresponding to partitions of $n$, that is, $$\left(\chi^{\lambda}_{\left\{f_1, \dots, f_r\right\}} \mid \lambda = \left(n_1,\dots,n_r\right) \vdash n; f_j \text{ is a Frobenius orbit of a character of } \multiplicativegroup{\finiteFieldExtension{n_j}} \right),$$
forms an orthogonal basis for the space of class functions of $\GL_n\left(\finiteField\right)$ (see also \cite[Page 286 Example 1]{macdonald1998symmetric}).

By \cite[Formula (39)]{Green55}, \begin{equation}\label{eq:degree-of-basic-character}
	\chi^{\lambda}_{\left\{f_1, \dots, f_r\right\}}\left( \IdentityMatrix{n} \right) = \frac{\Phi_n\left(q\right)}{\prod_{j = 1}^{\infty}\left(1-q^j\right)^{\lambda\left(j\right)}},
\end{equation}
where $\lambda\left(j\right)$ is the number of times $j$ appears in $\lambda$. Moreover, the computations in \cite[Pages 423 and 433]{Green55} yield that if $n = ab$ and $x \in \GL_a\left(\finiteField\right)$ is a regular elliptic element with eigenvalues $\{\xi,\xi^{q},\dots,\xi^{q^{a-1}}\}$, where $\xi \in \multiplicativegroup{\finiteFieldExtension{a}}$, and if $\mu \vdash b$, then
$\chi^{\lambda}_{\left\{f_1, \dots, f_r\right\}}\left( J_{\mu}\left(x\right) \right) = 0$ unless $\lambda = \left(a n'_1, \dots, a n'_r\right)$ for some partition $\lambda' = \left(n'_1, \dots, n'_r\right) \vdash b$, and in this case %$$\chi^{\lambda}_{\left\{f_1, \dots, f_r\right\}}\left( J_{\mu}\left(x\right) \right) =  \frac{1}{z_{\lambda'}} Q_{\lambda'}^{\mu}\left(q^{a}\right) \left(\prod_{i=1}^{\infty} \lambda'\left(i\right)!\right) \cdot \prod_{j = 1}^r \left(\sum_{i = 0}^{a n'_j -1}\theta_j\left(\xi^{q^{i}}\right)\right).$$
\begin{equation}\label{eq:evaluation-of-basic-character-at-an-element-with-a-unique-eigenvalue}
	\chi^{\lambda}_{\left\{f_1, \dots, f_r\right\}}\left( J_{\mu}\left(x\right) \right) =  Q_{\lambda'}^{\mu}\left(q^{a}\right) \cdot \prod_{j = 1}^r \left(\sum_{i = 0}^{a -1}\theta_j\left(\xi^{q^{i}}\right)\right),
\end{equation}
where $Q_{\lambda'}^{\mu}\left(q\right)$ is Green's polynomial \cite[Pages 246-248]{macdonald1998symmetric}.

Let $f'_1, \dots, f'_{r'}$ be all the Frobenius orbits that appear in the multiset $\left\{f_1,\dots,f_r\right\}$ without repetition, and let $m_j$ be the number of orbits among $f_1, \dots, f_r$ that equal $f'_j$. By \cite[Lemma 6.3]{Green55}, \begin{equation}\label{eq:norm-of-a-basic-character}
	\innerproduct{\chi^{\lambda}_{\left\{f_1, \dots, f_r\right\}}}{\chi^{\lambda}_{\left\{f_1, \dots, f_r\right\}}} = \left(\prod_{j=1}^{r} \frac{n_j}{\deg f_j}\right) \left(\prod_{i=1}^{r'} m_i!\right).
\end{equation}
Finally, we remark that $\chi^{\lambda}_{\left\{f_1, \dots, f_r\right\}}$ is a linear combination of $\trace \pi$, where $\pi$ goes over all irreducible representations of $\GL_n\left(\finiteField\right)$ with a certain cuspidal support, which we will describe shortly. To describe this cuspidal support, write $f_1 = f\left(\theta'_1 \circ \FieldNorm{n_1}{a_1}\right)$, $\dots$, $f_{r} = f\left(\theta'_{r} \circ \FieldNorm{n_{r}}{a_{r}}\right)$, where $n_j = a_j b_j$ and $\theta'_j \colon \multiplicativegroup{\finiteFieldExtension{a_j}} \to \multiplicativegroup{\cComplex}$ is a regular character. For each $j$, the Frobenius orbit $f\left(\theta'_j\right)$ corresponds to an irreducible cuspidal representation $\sigma_j$ of $\GL_{a_j}\left(\finiteField\right)$. Then $\chi^{\lambda}_{\left\{f_1, \dots, f_r\right\}}$ is a linear combination of characters of irreducible representations of $\GL_n\left(\finiteField\right)$, each having the same cuspidal support, which is the multiset formed by summing the following $r$ multisets: $b_1$ copies of $\sigma_1$, $b_2$ copies of $\sigma_2$, $\dots$, and $b_r$ copies of $\sigma_r$.

\subsection{Reduction of twisted matrix Kloosterman sums to twisted Kloosterman sums and Green polynomials}

In this section, we use Green's basic characters in order to reduce $\ExoticKloosterman\left(\alpha, \fieldCharacter, J_{\mu}\left(x\right)\right)$ for an elliptic element $x \in \GL_a\left(\finiteField\right)$ and $\mu \vdash b$, to an expression involving twisted Kloosterman sums of eigenvalues of $x$ and Green polynomials. The final expression is given in \Cref{thm:matrix-kloosterman-sum-at-jordan-using-green-functions}. We make use of \Cref{prop:kloosterman-character-identity} to find the expansion of $\ExoticKloosterman\left(\alpha, \fieldCharacter, y\right)$ in terms of Green's basic characters and Gauss sums. Then, we substitute $y = J_{\mu}\left(x\right)$ and make use of formula \eqref{eq:evaluation-of-basic-character-at-an-element-with-a-unique-eigenvalue} for the value of Green's basic characters at $J_{\mu}\left(x\right)$.

\subsubsection{Expansion of twisted matrix Kloosterman sums in terms of basic characters}

Let $\alpha \colon \left(\multiplicativegroup{\finiteField}\right)^k \to \multiplicativegroup{\cComplex}$ be a character. Write $\alpha = \alpha_1 \times \dots \times \alpha_k$, where $\alpha_1,\dots,\alpha_k \colon \multiplicativegroup{\finiteField} \to \multiplicativegroup{\cComplex}$ are characters. Consider the class function $\KloostermanSumClassFunction \colon \GL_n\left(\finiteField\right) \to \multiplicativegroup{\cComplex}$ defined by $\KloostermanSumClassFunction\left(x\right) = \ExoticKloosterman\left(\alpha, \fieldCharacter, x\right)$. Then we have $$\KloostermanSumClassFunction\left(x\right) = \sum_{ \lambda = \left(n_1,\dots,n_r\right) \vdash n} \sum_{ \{f_1,\dots,f_r\}} \frac{\innerproduct{\KloostermanSumClassFunction}{\conjugate{\chi^{\lambda}_{\left\{f_1, \dots, f_r\right\}}}}}{\innerproduct{\chi^{\lambda}_{\left\{f_1, \dots, f_r\right\}}}{\chi^{\lambda}_{\left\{f_1, \dots, f_r\right\}}}} \conjugate{\chi^{\lambda}_{\left\{f_1, \dots, f_r\right\}}\left(x\right)}.$$

Let us evaluate $\innerproduct{\KloostermanSumClassFunction}{\conjugate{\chi^{\lambda}_{\left\{f_1, \dots, f_r\right\}}}}$. Write $$\chi^{\lambda}_{\left\{f_1, \dots, f_r\right\}} = \sum_{\pi} c_\pi \trace \pi,$$
where $\pi$ goes over all irreducible representations of $\GL_n\left(\finiteField\right)$ that have cuspidal support as described at the end of \Cref{subsec:green-basic-characters}, and $c_\pi \in \cComplex$. Then $$ {\innerproduct{\KloostermanSumClassFunction}{\conjugate{\chi^{\lambda}_{\left\{f_1, \dots, f_r\right\}}}}} = \frac{1}{\sizeof{\GL_n\left(\finiteField\right)}} \sum_{\pi} \sum_{g_1,\dots,g_k \in \GL_n\left(\finiteField\right)} {c_{\pi}} \left(\prod_{i=1}^k \alpha_i\left(\det g_i\right)\right) \fieldCharacter\left(\trace \sum_{i=1}^k g_i\right) {\trace \pi\left(g_1\dots g_k\right)}.$$
By \eqref{eq:trace-of-product-of-non-abelian-gauss-sum}, we get that ${\innerproduct{\KloostermanSumClassFunction}{\conjugate{\chi^{\lambda}_{\left\{f_1, \dots, f_r\right\}}}}}$ is given by \begin{align*}
	& \frac{q^{\frac{k n^2}{2}}}{\sizeof{\GL_n\left(\finiteField\right)}} \left(\prod_{i=1}^k \GKGaussSumScalar{\pi'}{\alpha_i}{\fieldCharacter}\right) \sum_{\pi} {c_\pi} \trace \pi\left(\IdentityMatrix{n}\right) \\
	=& \frac{q^{\frac{k n^2}{2}}}{\sizeof{\GL_n\left(\finiteField\right)}}  \left(\prod_{i=1}^k \GKGaussSumScalar{\pi'}{\alpha_i}{\fieldCharacter}\right) \chi^{\lambda}_{\left\{f_1, \dots, f_r\right\}}\left(\IdentityMatrix{n}\right),
\end{align*}
where $\pi'$ is any irreducible representation of $\GL_n\left(\finiteField\right)$ with cuspidal support as described at the end of \Cref{subsec:green-basic-characters}. Let $\sigma_1, \dots, \sigma_r$ be as at the end of \Cref{subsec:green-basic-characters}. By \Cref{thm:kondos-formula}, we have $$q^{\frac{n}{2}} \GKGaussSumScalar{\pi'}{\alpha_i}{\fieldCharacter} = \left(-1\right)^n \prod_{j=1}^r \left( \GaussSumCharacter{a_j}{\theta'_j}{\alpha_i}{\fieldCharacter}\right)^{b_j} = \left(-1\right)^n \prod_{j=1}^r \GaussSumCharacter{a_j b_j}{\theta'_j \circ \FieldNorm{a_j b_j}{a_j}}{\alpha_i}{\fieldCharacter}.$$ 
where we used the Hasse--Davenport relation for the last step. Combining this with \eqref{eq:degree-of-basic-character}, we have
$${\innerproduct{\KloostermanSumClassFunction}{\conjugate{\chi^{\lambda}_{\left\{f_1, \dots, f_r\right\}}}}} = \left(-1\right)^{\left(k-1\right)n} \frac{q^{\left(k-1\right)\binom{n}{2}}}{\prod_{j = 1}^{\infty}\left(1-q^j\right)^{\lambda\left(j\right)}} \cdot \prod_{i = 1}^k \prod_{j=1}^{r}\GaussSumCharacter{n_j}{\theta_j}{\alpha_i}{\fieldCharacter}.$$

Hence, we obtained the formula
\begin{equation} \label{eq:matrix-kloosterman-sum-in-terms-of-basic-characters}
	\begin{split}
		\KloostermanSumClassFunction\left(x\right) &= \left(-1\right)^{\left(k-1\right)n} \sum_{ \lambda = \left(n_1,\dots,n_r\right) \vdash n} \sum_{ \{f_1,\dots,f_r\}} \frac{\conjugate{\chi^{\lambda}_{\left\{f_1, \dots, f_r\right\}}\left(x\right)}}{\innerproduct{\chi^{\lambda}_{\left\{f_1, \dots, f_r\right\}}}{\chi^{\lambda}_{\left\{f_1, \dots, f_r\right\}}}} \frac{q^{\left(k-1\right) \binom{n}{2} }}{\prod_{j = 1}^{\infty}\left(1-q^j\right)^{\lambda\left(j\right)}} \\
		& \times \prod_{i = 1}^k \prod_{j=1}^{r} \GaussSumCharacter{n_j}{f_j}{\alpha_i}{\fieldCharacter},
	\end{split}
\end{equation}

where $\GaussSumCharacter{n_j}{f_j}{\alpha_i}{\fieldCharacter} = \GaussSumCharacter{n_j}{\theta_j}{\alpha_i}{\fieldCharacter}$.

\subsubsection{Evaluation at $J_{\mu}\left(x\right)$}

We would like to substitute $J_{\mu}\left(x\right)$ in \eqref{eq:matrix-kloosterman-sum-in-terms-of-basic-characters}, where $n = ab$,  $\mu \vdash b$, and $x \in \GL_a\left(\finiteField\right)$ is a regular elliptic element with eigenvalues $\{\xi, \xi^q, \dots, \xi^{q^{a-1}}\}$.

We already have an expression for the inner product $\innerproduct{\chi^{\lambda}_{\left\{f_1, \dots, f_r\right\}}}{\chi^{\lambda}_{\left\{f_1, \dots, f_r\right\}}}$ in \eqref{eq:norm-of-a-basic-character} and an expression for $\chi^{\lambda}_{\left\{f_1, \dots, f_r\right\}}\left(J_{\mu}\left(x\right)\right)$ in \eqref{eq:evaluation-of-basic-character-at-an-element-with-a-unique-eigenvalue}. We would like to simplify \eqref{eq:matrix-kloosterman-sum-in-terms-of-basic-characters} by replacing the inner sum over the Frobenius orbits $\left\{f_1,\dots,f_r\right\}$, with a sum over all characters. The simplified version is given in the following theorem.

\begin{theorem}
	Let $\lambda = \left(n_1,\dots,n_r\right) \vdash n$ be a partition. Denote 
	\begin{align*}
			\KloostermanSumClassFunction_{\lambda}\left( J_{\mu}\left(x\right) \right) &= \sum_{ \{f_1,\dots,f_r\}} \frac{\conjugate{\chi^{\lambda}_{\left\{f_1, \dots, f_r\right\}}\left(J_{\mu}\left(x\right)\right)}}{\innerproduct{\chi^{\lambda}_{\left\{f_1, \dots, f_r\right\}}}{\chi^{\lambda}_{\left\{f_1, \dots, f_r\right\}}}} \frac{q^{\left(k-1\right)\binom{n}{2}}}{\prod_{j = 1}^{\infty}\left(1-q^j\right)^{\lambda\left(j\right)}} \prod_{i = 1}^k \prod_{j=1}^{r}\GaussSumCharacter{n_j}{f_j}{\alpha_i}{\fieldCharacter},
	\end{align*}
	where the sum is over all multisets of Frobenius orbits $\left\{f_1,\dots,f_r\right\}$, such that $f_j$ is a Frobenius orbit of a character of $\multiplicativegroup{\finiteFieldExtension{n_j}}$.
	
	Then $\KloostermanSumClassFunction_{\lambda}\left(J_{\mu}\left(x\right)\right) = 0$ unless $\lambda = \left(a n'_1, \dots, a n'_r\right)$, where $\lambda' = \left(n'_1,\dots,n'_r\right) \vdash b$, and in this case
	$$ \KloostermanSumClassFunction_{\lambda}\left( J_{\mu}\left(x\right) \right) = \left(-1\right)^{\ell\left(\lambda\right) \left(k-1\right)} q^{\left(k-1\right) \binom{n}{2}} \cdot \frac{Q_{\lambda'}^{\mu}\left(q^a\right)}{z_{\lambda'}} \prod_{j=1}^{r} \ExoticKloosterman_{a n'_j}\left(\alpha, \fieldCharacter, \xi\right).$$
	Here $z_{\lambda'}  = \prod_{i = 1}^{\infty} i^{\lambda'\left(i\right)} \lambda'\left(i\right)!$, where $\lambda'\left(i\right)$ is the number of times $i$ appears in $\lambda'$.
\end{theorem}
\begin{proof}
	We have that $\chi^{\lambda}_{\left\{f_1, \dots, f_r\right\}}\left(J_{\mu}\left(x\right)\right)$ is zero unless $\lambda$ is of the form given in the theorem. In this case, by \eqref{eq:evaluation-of-basic-character-at-an-element-with-a-unique-eigenvalue},
	\begin{equation}\label{eq:expression-of-K-lambda-in-terms-of-K-prime-lambda}
		\KloostermanSumClassFunction_{\lambda}\left(J_{\mu}\left(x\right)\right) = \left(-1\right)^{\ell \left(\lambda\right)} \frac{q^{\left(k-1\right)\binom{n}{2}}}{\prod_{j = 1}^{\infty}\left(q^j-1\right)^{\lambda\left(j\right)}} Q_{\lambda'}^{\mu}\left(q^{a}\right)  \cdot \KloostermanSumClassFunction'_{\lambda}\left(J_{\mu}\left(x\right)\right),
	\end{equation}
	where
	\begin{equation*}
		\begin{split}
			\KloostermanSumClassFunction'_{\lambda}\left(J_{\mu}\left(x\right)\right) =& \sum_{\left\{f_1, \dots, f_r\right\}} \frac{1}{\innerproduct{\chi^{\lambda}_{\left\{f_1, \dots, f_r\right\}}}{\chi^{\lambda}_{\left\{f_1, \dots, f_r\right\}}}} \prod_{j = 1}^r \left( \sum_{d = 0}^{a -1} \theta_j^{-q^d}\left(\xi\right) \prod_{i=1}^k \GaussSumCharacter{n_j}{\theta_j}{\alpha_i}{\fieldCharacter}\right).
		\end{split}
	\end{equation*}
	Since there are $\deg \theta_j$ elements in the Frobenius orbit $\{\theta_j^{q^d} \mid d \ge 0\}$, we may rewrite $\KloostermanSumClassFunction'_{\lambda}\left(J_{\mu}\left(x\right)\right)$ as
	\begin{equation*}
		\begin{split}
			\sum_{\left\{f_1, \dots, f_r\right\}} \frac{1}{\innerproduct{\chi^{\lambda}_{\left\{f_1, \dots, f_r\right\}}}{\chi^{\lambda}_{\left\{f_1, \dots, f_r\right\}}}} \prod_{j = 1}^r \frac{a}{\deg \theta_j} \left( \sum_{d = 0}^{\deg \theta_j -1} \theta_j^{-q^d}\left(\xi\right) \prod_{i=1}^k \GaussSumCharacter{n_j}{\theta_j}{\alpha_i}{\fieldCharacter} \right).
		\end{split}
	\end{equation*}	
	The factor $\prod_{j=1}^r \frac{1}{\deg \theta_j}$ also appears in $\innerproduct{\chi^{\lambda}_{\left\{f_1, \dots, f_r\right\}}}{\chi^{\lambda}_{\left\{f_1, \dots, f_r\right\}}}$ (see \eqref{eq:norm-of-a-basic-character}). Hence it cancels, and we get that $\KloostermanSumClassFunction'_{\lambda}\left(J_{\mu}\left(x\right)\right)$ is given by
	\begin{equation*}
		\begin{split}
			\frac{a^r}{\left(\prod_{i=1}^r n_i \right) \left(\prod_{j=1}^{r'} m_j!\right)} \sum_{\left\{f_1, \dots, f_r\right\}}  \prod_{j = 1}^r \left(\sum_{d = 0}^{\deg \theta_j -1} \theta_j^{-q^d}\left(\xi\right) \prod_{i=1}^k \GaussSumCharacter{n_j}{\theta_j}{\alpha_i}{\fieldCharacter} \right),
		\end{split}
	\end{equation*}
	where if $f'_1, \dots, f'_{r'}$ are all the different Frobenius orbit among $f_1, \dots, f_r$, without repetition, then $m_j$ is the number of times $f'_j$ appears in $f_1, \dots, f_r$. Notice that $\frac{n_i}{a} = n'_i$, and therefore this sum is the same as 
	\begin{equation}\label{eq:last-summation-over-frobenius-orbits}
		\begin{split}
			\frac{1}{\left(\prod_{i=1}^r n'_i \right) \left(\prod_{j=1}^{r'} m_j!\right)} \sum_{\left\{f_1, \dots, f_r\right\}}  \prod_{j = 1}^r \left(\sum_{d = 0}^{\deg \theta_j -1} \theta_j^{-q^d}\left(\xi\right) \prod_{i=1}^k \GaussSumCharacter{n_j}{\theta_j}{\alpha_i}{\fieldCharacter} \right).
		\end{split}
	\end{equation}	
	
	We claim that $\KloostermanSumClassFunction'_{\lambda}\left(J_{\mu}\left(x\right)\right)$ can be written as
	\begin{equation}\label{eq:summation-over-character-group}
	\begin{split}
		\frac{1}{z_{\lambda'}} \sum_{\theta_1 \in \charactergroup{n_1}, \dots, \theta_r \in \charactergroup{n_r}} \prod_{j = 1}^r \left(\theta^{-1}_j\left(\xi\right) \prod_{i=1}^k \GaussSumCharacter{n_j}{\theta_j}{\alpha_i}{\fieldCharacter} \right).
	\end{split}
\end{equation}	
	In order to see this, we need to count how many times a multiset $\left\{f_1,\dots,f_r\right\}$ appears as a multiset $\{f\left(\theta_1\right), f\left(\theta_2\right), \dots, f\left(\theta_r\right)\}$ for Frobenius orbits $f\left(\theta_1\right)$, $\dots$, $f\left(\theta_r\right)$ corresponding to characters as in \eqref{eq:summation-over-character-group}. 
	Write $\lambda' = \left(1^{\lambda'\left(1\right)}, 2^{\lambda'\left(2\right)}, \dots\right)$ and for every $j$, let $f_{j 1}, \dots, f_{j \lambda'\left(j\right)}$ be the Frobenius orbits of characters in $\charactergroup{aj}$, among $f_1,\dots,f_r$. Let $f'_{j 1}, \dots, f'_{j l_{j}}$ be the different Frobenius orbits among $f_{j 1},\dots,f_{j \lambda'\left(j\right)}$ and let $m_{j i}$ be the number of times $f'_{j i}$ appears among $f_{j 1}, \dots, f_{j \lambda'\left(j\right)}$. The number of ways that $\left\{f_1,\dots,f_r\right\}$ appears as a multiset as above is the following product of multinomial coefficients
	$$ \prod_{j = 1}^{\infty} \binom{\lambda'\left(j\right)}{m_{j 1}, \dots, m_{j l_j}}.$$
	Notice that $$\frac{1}{z_{\lambda'}} \prod_{j = 1}^{\infty} \binom{\lambda'\left(j\right)}{m_{j 1}, \dots, m_{j l_j}} = \frac{1}{\left(\prod_{i = 1}^r n'_i\right) \cdot \left(\prod_{j=1}^{\infty} \prod_{i=1}^{l_{j}} m_{j i}!\right)} = \frac{1}{\left(\prod_{i=1}^r n'_i \right) \left(\prod_{j=1}^{r'} m_j!\right)},$$
	and therefore, \eqref{eq:last-summation-over-frobenius-orbits} is the same as \eqref{eq:summation-over-character-group}.
	
	We are now able to evaluate the sum in \eqref{eq:summation-over-character-group}. We have that $\KloostermanSumClassFunction'_{\lambda}\left(J_{\mu}\left(x\right)\right)$ is given by \begin{equation*}
		\begin{split}
			& \left(-1\right)^{kr} \frac{1}{z_{\lambda'}} \sum_{\theta_1 \in \charactergroup{n_1}, \dots, \theta_r \in \charactergroup{n_r}} \\
			& \times \prod_{j = 1}^r \left( \sum_{t_{j1},\dots,t_{jk} \in \multiplicativegroup{\finiteFieldExtension{n_j}}} \left(\prod_{i=1}^k \alpha_i\left(\FieldNorm{n_j}{1} \left(t_{ji}\right)\right)\right) \theta_j\left( \xi^{-1} \prod_{i=1}^k t_{ji} \right) \fieldCharacter_{n_j}\left(\sum_{i=1}^k t_{ji}\right) \right).
		\end{split}
	\end{equation*}
	As before, after changing the order of summation, the inner sum over $\theta_j \in \multiplicativegroup{\finiteFieldExtension{n_j}}$ vanishes, unless $\xi^{-1} \prod_{i = 1}^k t_{j i} = 1$, in which case it evaluates to $\sizeof{\multiplicativegroup{\finiteFieldExtension{n_j}}}$. Hence, $$\KloostermanSumClassFunction'_{\lambda}\left(J_{\mu}\left(x\right)\right)= \left(-1\right)^{k \ell\left(\lambda\right)} \frac{1}{z_{\lambda'}} \prod_{j = 1}^r \sizeof{\multiplicativegroup{\finiteFieldExtension{n_j}}} \ExoticKloosterman_{n_j}\left(\alpha, \fieldCharacter, \xi\right).$$
	Substituting this in \eqref{eq:expression-of-K-lambda-in-terms-of-K-prime-lambda} yields the result.
\end{proof}

As a result this computation and of \eqref{eq:matrix-kloosterman-sum-in-terms-of-basic-characters}, we obtain the following theorem.

\begin{theorem}\label{thm:matrix-kloosterman-sum-at-jordan-using-green-functions}
	Let $n = ab$ and let $x \in \GL_a\left(\finiteField\right)$ be a regular elliptic element with eigenvalues $\{\xi, \xi^q, \dots, \xi^{q^{a-1}}\}$. Then for any $k > 1$ and any character $\alpha \colon \left(\multiplicativegroup{\finiteField}\right)^k \to \multiplicativegroup{\cComplex}$, and for any partition $\mu \vdash b$, we have
	$$\ExoticKloosterman\left(\alpha, \fieldCharacter, J_{\mu}\left(x\right) \right) = \left(-1\right)^{\left(k-1\right)n} q^{\left(k-1\right) \binom{n}{2}} \sum_{\lambda = \left(b_1,\dots,b_r\right) \vdash b} \left(-1\right)^{\ell\left(\lambda\right)\left(k-1\right)} \frac{Q_{\lambda}^{\mu}\left(q^a\right)}{z_{\lambda}} \prod_{j=1}^{r} \ExoticKloosterman_{a b_j}\left(\alpha, \fieldCharacter, \xi\right). $$
\end{theorem}

\subsection{Expression in terms of modified Hall--Littlewood polynomials and twisted Kloosterman sheaves}

In this section, we use modified Hall--Littlewood polynomials and twisted Kloosterman sheaves in order to reformulate \Cref{thm:matrix-kloosterman-sum-at-jordan-using-green-functions}. We start with recalling results regarding the number of flags of type $\lambda$ fixed by a Jordan matrix $J_{\mu}\left(\xi\right)$. We then discuss the definition of modified Hall--Littlewood polynomials and obtain formulas related to the previous discussion about flags fixed by a Jordan matrix, and a formula similar to the one appearing in \Cref{thm:matrix-kloosterman-sum-at-jordan-using-green-functions}. We then discuss the notion of twisted Kloosterman sheaves defined by Katz \cite{katz2016gauss}, and the properties of the characteristic polynomials of the action of the geometric Frobenius on a stalk. Finally, in the last section, we use these results to express $\ExoticKloosterman\left(\alpha, \fieldCharacter, J_{\mu}\left(x\right)\right)$ using modified Hall--Littlewood polynomials and the roots of $\ExoticKloosterman^{\finiteFieldExtension{a}}\left(\alpha, \fieldCharacter\right)$ at $\xi$.

\subsubsection{Weak flags corresponding to a composition type}

Let $m$ be a positive integer. A sequence $\left(m_1, \dots, m_t\right)$ is a \emph{weak composition} of $m$ if $m_1 + \dots + m_t = m$ and $m_1, \dots, m_t \ge 0$. A \emph{weak flag of type $\left(m_1,\dots,m_t\right)$} in the vector space $\finiteFieldExtension{a}^m$ is a sequence of $\finiteFieldExtension{a}$-subspaces \begin{equation}\label{eq:flag-sequence}
	\mathcal{F}   \colon 0 \subseteq W_1 \subseteq W_2 \subseteq \dots \subseteq W_t = \finiteFieldExtension{a}^m,\end{equation}
such that for every $j$, $$ \dim W_j = m_1 + \dots + m_j.$$ 
The number $t$ is the called the \emph{length} of the weak flag $\mathcal{F}$. A \emph{flag} $\mathcal{F}$ is a weak flag of type $\left(m_1, \dots, m_t\right)$, for some composition $\left(m_1, \dots, m_t\right)$ of $m$, that is, $m_1, \dots, m_t > 0$ and $m_1 + \dots + m_t = m$.

The group $\GL_m\left(\finiteFieldExtension{a}\right)$ acts on the set of weak flags of type $\left(m_1,\dots,m_t\right)$, as follows. If $g \in \GL_m\left(\finiteFieldExtension{a}\right)$ and $\mathcal{F}$ is a weak flag of type $\left(m_1,\dots,m_t\right)$ corresponding to the sequence \eqref{eq:flag-sequence}, then $g \mathcal{F}$ is the weak flag defined by the sequence
$$ g \mathcal{F} \colon 0 \subseteq gW_1 \subseteq gW_2 \subseteq \dots \subseteq gW_t = \finiteFieldExtension{a}^m. $$

For an element $g \in \GL_m\left(\finiteFieldExtension{a}\right)$, let $$ \mathcal{X}^{\finiteFieldExtension{a}}_{\left(m_1,\dots,m_t\right)}\left(g\right) = \left\{ \mathcal{F} \mid \mathcal{F} \text{ is a weak flag of type } \left(m_1,\dots,m_t\right) \text{ in } \finiteFieldExtension{a}^m \mid g \mathcal{F} = \mathcal{F} \right\}.$$ By \cite[Proposition 2.3]{Ram2023}, for any two weak compositions $\left(m_1,\dots,m_t\right)$ and $\left(m'_1,\dots,m'_{t'}\right)$ such that the trimming of zeros and reordering in decreasing order of $\left(m_1,\dots,m_t\right)$ and $\left(m'_1,\dots,m'_{t'}\right)$ results in the same partition, we have the equality
\begin{equation*}
	\sizeof{\mathcal{X}_{\left(m_1,\dots,m_t\right)}^{\finiteFieldExtension{a}}\left(g\right)} = \sizeof{\mathcal{X}_{\left(m'_1,\dots,m'_{t'}\right)}^{\finiteFieldExtension{a}}\left(g\right)}.
\end{equation*}

For any two partitions $\mu \vdash m$ and $\lambda \vdash m$ denote, as in \cite[Theorem 3.1]{Kirillov00},
$$\mathcal{P}_{\mu, \lambda}\left(t\right) = \sum_{\rho \vdash m} K_{\rho, \lambda} K_{\rho, \mu}\left(t\right),$$
where $K_{\rho, \lambda}$ is the Kostka number \cite[Page 101]{macdonald1998symmetric} and $K_{\rho, \mu}\left(t\right)$ is the Kostka--Foulkes polynomial \cite[Page 242]{macdonald1998symmetric} (recall that $K_{\rho, \lambda} = K_{\rho, \lambda}\left(1\right)$). By \cite[Section 1.4]{Kirillov00}, we have
\begin{theorem}\label{thm:number-of-flags-fixed-by-a-jordan-matrix}For any two partitions $\lambda \vdash m$ and $\mu \vdash m$, and any $\xi \in \multiplicativegroup{\finiteFieldExtension{a}}$,
	$$\sizeof{\mathcal{X}_{\lambda}^{\finiteFieldExtension{a}}\left(J_{\mu}\left(\xi\right)\right)} =  q^{a \mathfrak{n}\left(\mu\right)} \mathcal{P}_{\mu, \lambda}\left(q^{-a}\right),$$
	where the partition $\lambda \vdash m$ is regarded as a composition of $m$ (see the definition of $\mathfrak{n}\left(\mu\right)$ below).
\end{theorem} 

\subsubsection{Modified Hall--Littlewood polynomials}\label{subsec:hall-littlewood}

For any positive integer $b$ and any partition $\mu \vdash b$, the modified Hall--Littlewood polynomial with $k$ variables is defined by the formula \cite[Section 3.4.7]{Haiman2003}, \cite[Page 12]{DesarmenienLeclercThibon1994}
$$ \mathrm{\tilde{H}}_{\mu}\left(x_1,\dots,x_k;t\right) = t^{\mathfrak{n}\left(\mu\right)} \sum_{\rho \vdash b} K_{\rho, \mu}\left(t^{-1}\right) s_{\rho}\left(x_1,\dots,x_k\right),$$
where $s_{\rho}\left(x_1,\dots,x_k\right)$ is the Schur polynomial with $k$ variables corresponding to the partition $\rho \vdash b$, and for $\mu = \left(b'_1, \dots, b'_t\right) \vdash b$,
$$\mathfrak{n}\left(\mu\right) = \sum_{j=1}^t \left(j-1\right) b'_j.$$
Let us mention that in \cite[Formula (12)]{DesarmenienLeclercThibon1994}, the authors first consider the polynomial $$Q'_\mu\left(x_1,\dots,x_k, t\right) = \sum_{\rho \vdash b} K_{\rho, \mu}\left(t\right) s_{\rho}\left(x_1,\dots,x_k\right) = t^{\mathfrak{n}\left(\mu\right)} \mathrm{\tilde{H}}_{\mu}\left(x_1,\dots,x_k;t^{-1}\right),$$
but later on Page 12, they consider $$\tilde{Q}'_{\mu}\left(x_1,\dots,x_k;t\right) = t^{\mathfrak{n}\left(\mu\right)} Q'_\mu\left(x_1,\dots,x_k, t^{-1}\right)$$ which agrees with $\mathrm{\tilde{H}}_{\mu}\left(x_1,\dots,x_k;t\right)$. In \cite[Corollary 3.4.12]{Haiman2003}, $\mathrm{H}_{\mu}$ is the same as $Q'_\mu$ defined above, and later in \cite[Section 3.4.7]{Haiman2003}, $\mathrm{\tilde{H}}_{\mu}$ is the same as the definition here.

We will give two other expressions for $\mathrm{\tilde{H}}_{\mu}\left(x_1,\dots,x_k;t\right)$. The first one expresses it in terms of weak flags fixed by a Jordan matrix of the form $J_{\mu}\left(\xi\right)$. By \cite[Page 101]{macdonald1998symmetric}, we have that for any partition $\rho \vdash b$,
$$s_{\rho}\left(x_1,\dots,x_k\right) = \sum_{\lambda \vdash b} K_{\rho, \lambda} m_{\lambda}\left(x_1,\dots, x_k\right).$$
where $m_{\lambda}\left(x_1,\dots,x_k\right)$ is the monomial symmetric polynomial corresponding to the partition $\lambda$.
Therefore,
$$\mathrm{\tilde{H}}_{\mu}\left(x_1,\dots,x_k;t\right) = t^{\mathfrak{n}\left(\mu\right)} \sum_{\lambda \vdash b} \mathcal{P}_{\mu, \lambda}\left(t^{-1}\right) m_{\lambda}\left(x_1,\dots,x_k\right).$$
In the special case where $t = q^a$ for a positive integer $a$, we get by \Cref{thm:number-of-flags-fixed-by-a-jordan-matrix} that for any $\xi \in \multiplicativegroup{\finiteFieldExtension{a}}$ and any $\mu \vdash b$,
\begin{equation}\label{eq:sum-of-monomials-with-fixed-flags-coefficients}
	\mathrm{\tilde{H}}_{\mu}\left(x_1,\dots,x_k;q^a\right) = \sum_{\lambda \vdash b} \# \left\{ \mathcal{F} \text{ flag of type } \lambda \text{ in } \finiteFieldExtension{a}^b \mid J_{\mu}\left(\xi\right) \mathcal{F} = \mathcal{F} \lambda \right\} \cdot m_{\lambda}\left(x_1,\dots,x_k\right).
\end{equation}

The second expression we give for $\mathrm{\tilde{H}}_{\mu}\left(x_1,\dots,x_k;t\right)$, expresses it in terms of power sum polynomials. For any $m$, let $p_m$ be the $m$-th power sum polynomial with $k$ variables
$$ p_m\left(x_1,\dots,x_k\right) = \sum_{i=1}^k x_i^m,$$ and for any partition $\lambda = \left(b_1, \dots, b_t\right) \vdash b$, denote $$p_\lambda\left(x_1,\dots,x_k\right) = \prod_{j=1}^t p_{b_j}\left(x_1,\dots,x_k\right).$$
Recall the following identity \cite[Page 177, Formula (9.5)]{macdonald1998symmetric}:
$$s_{\rho}\left(x_1,\dots,x_k\right) = \sum_{\lambda \vdash b} \frac{\chi^{\rho}_{\lambda}}{z_{\lambda}} p_{\lambda}\left({x_1,\dots,x_k}\right),$$
where $\rho$ is a partition of $b$, and where for any $\lambda \vdash b$,  $\chi_{\lambda}^{\rho}$ is the value of the character of the symmetric group $\SymmetricGroup_b$ corresponding to the partition $\rho$ at the conjugacy class corresponding to cycles of type $\lambda$. Then we have
$$ \mathrm{\tilde{H}}_{\mu}\left(x_1,\dots,x_k;t\right) = t^{\mathfrak{n}\left(\mu\right)} \sum_{\lambda \vdash b} \left(\sum_{\rho \vdash b} \chi_{\lambda}^{\rho} K_{\rho, \mu}\left(t^{-1}\right) \right) \frac{1}{z_\lambda}p_{\lambda}\left(x_1,\dots,x_k\right).$$
By \cite[Page 247, Formula (7.6')]{macdonald1998symmetric},
$$ X_{\lambda}^{\mu}\left(t\right) = \sum_{\rho \vdash b} \chi_{\lambda}^{\rho} K_{\rho, \mu}\left(t\right),$$
where \cite[Page 247, Formula (7.8)]{macdonald1998symmetric} $$Q_{\lambda}^{\mu}\left(t\right) = t^{\mathfrak{n}\left(\mu\right)} X_{\lambda}^{\mu}\left(t^{-1}\right).$$
Hence, 
\begin{equation}\label{eq:hall-littlewood-polynomial-in-terms-of-power-sum-polynomials}
	\mathrm{\tilde{H}}_{\mu}\left(x_1,\dots,x_k;t\right) = \sum_{\lambda \vdash b} \frac{Q_{\lambda}^{\mu}\left(t\right)}{z_\lambda} \cdot p_{\lambda}\left(x_1,\dots,x_k\right).
\end{equation}
This formula is very similar to the one in \Cref{thm:matrix-kloosterman-sum-at-jordan-using-green-functions}. We will use the roots of the Kloosterman sheaf $\ExoticKloosterman^{\finiteFieldExtension{a}}\left(\alpha, \fieldCharacter\right)$ at $\xi$ to establish a relation between these formulas.

\subsubsection{The roots of $\ExoticKloosterman^{\finiteFieldExtension{a}}\left(\alpha, \fieldCharacter\right)$ at $\xi$}\label{subsec:roots-of-frobenius}

For any $a \ge 1$, consider the following diagram of schemes over $\finiteFieldExtension{a}$:
$$ \xymatrix{ & \multiplcativeScheme^k \ar[ld]_{\mathrm{mul}}  \ar[rd]^{\mathrm{add}} \\
	\multiplcativeScheme & & \affineLine }.$$
Let $\ell$ be a prime number different from the characteristic of $\finiteField$. Fix a non-trivial additive character $\fieldCharacter \colon \multiplicativegroup{\finiteFieldExtension{a}} \to \multiplicativegroup{\ladicnumbers}$. Let $\alpha \colon \left(\multiplicativegroup{\finiteField}\right)^k \to \multiplicativegroup{\ladicnumbers}$ be a multiplicative character. For any $a \ge 1$, the additive character $\fieldCharacter_{a}$ gives rise to an Artin--Schrier local system $\artinScrier_{\fieldCharacter_{a}}$ on $\affineLine$. The multiplicative character $\alpha \circ \FieldNorm{a}{1}$ gives rise to a Kummer local system $\mathcal{L}_{\alpha \circ \FieldNorm{a}{1}}$ on $\multiplcativeScheme$. Let $\ExoticKloosterman^{\finiteFieldExtension{a}}\left(\alpha, \fieldCharacter\right)$ be the twisted Kloosterman sheaf corresponding to the characters $\alpha$ and $\fieldCharacter$, that is,
$$ \ExoticKloosterman^{\finiteFieldExtension{a}}\left(\alpha, \fieldCharacter\right) = \convolutionWithCompactSupport \mathrm{mul}_{!} \left(\mathrm{add}^{\ast} \artinScrier_{\fieldCharacter_{a}} \otimes \mathcal{L}_{\alpha \circ \FieldNorm{a}{1}}\right)\left[k-1\right].$$ 

We will be mostly interested in $\Frobenius_{\xi} \mid \ExoticKloosterman^{\finiteFieldExtension{a}}\left(\alpha, \fieldCharacter\right)_{\xi}$, that is, the action of the geometric Frobenius on the stalk at $\xi \in \multiplicativegroup{\finiteFieldExtension{a}}$. Fix an embedding $\ladicnumbers \hookrightarrow \cComplex$. By \cite[Page 49, Theorem 4.1.1, part 2]{katz2016gauss}, we have that for any $m \ge 1$, \begin{equation}\label{eq:trace-of-frobenius-to-the-m}
	\trace \left(\Frobenius^m_{\xi} \mid \ExoticKloosterman^{\finiteFieldExtension{a}}\left(\alpha, \fieldCharacter\right)_{\xi}\right) = \left(-1\right)^{k-1} \ExoticKloosterman_{am}\left(\alpha, \fieldCharacter ,\xi\right).
\end{equation}

Let $\mathrm{Char}_{\Frobenius_{\xi}}\left(T\right)$ be the characteristic polynomial of $\Frobenius_{\xi} \mid \ExoticKloosterman^{\finiteFieldExtension{a}}\left(\alpha, \fieldCharacter\right)_{\xi}$. By \cite[Page 49, Theorem 4.1.1, part 1]{katz2016gauss}, it is a polynomial of degree $k$. Let $\omega_1, \dots, \omega_k \in \cComplex$ be its roots. We call $\omega_1, \dots, \omega_k$ \emph{the roots of $\ExoticKloosterman^{\finiteFieldExtension{a}}\left(\alpha, \fieldCharacter\right)$ at $\xi$}. By \cite[Page 49, Theorem 4.1.1, part 1]{katz2016gauss}, we have that $\abs{\omega_1} = \dots = \abs{\omega_k} = q^{\frac{a\left(k-1\right)}{2}}$.

We can rewrite \eqref{eq:trace-of-frobenius-to-the-m} as 
\begin{equation}
	\label{eq:power-sum-polynomial-m}
	p_m\left(\omega_1,\dots,\omega_k\right) = \left(-1\right)^{k-1} \ExoticKloosterman_{am}\left(\alpha, \fieldCharacter ,\xi\right).
\end{equation}

Alternatively, we can define the roots $\omega_1, \dots, \omega_k$ of $\ExoticKloosterman^{\finiteFieldExtension{a}}\left(\alpha, \fieldCharacter\right)$ at $\xi$ avoiding Kloosterman sheaves, using the $L$-function associated to the family $\left(\ExoticKloosterman_{am}\left(\alpha, \fieldCharacter, \xi\right)\right)_{m=1}^{\infty}$. To do so, define
$$ L\left(T, \ExoticKloosterman^{\finiteFieldExtension{a}}\left(\alpha, \fieldCharacter, \xi\right)\right) = \exp\left(\sum_{m=1}^{\infty} \ExoticKloosterman_{am}\left(\alpha,\fieldCharacter,\xi\right) \frac{T^m}{m}\right).$$
Then $L\left(T, \ExoticKloosterman^{\finiteFieldExtension{a}}\left(\alpha, \fieldCharacter, \xi \right)\right)^{\left(-1\right)^k}$ is a polynomial of degree $k$ with constant coefficient $1$ (see the introduction of \cite{fu2005functions}). We have
$$L\left(T, \ExoticKloosterman^{\finiteFieldExtension{a}}\left(\alpha, \fieldCharacter, \xi \right)\right)^{\left(-1\right)^k} = \prod_{i=1}^k \left(1 - \omega_i T \right),$$
and \eqref{eq:power-sum-polynomial-m} holds.

\subsubsection{Formula in terms of modified Hall--Littlewood polynomials}

Let $n = ab$ and let $x \in \GL_a\left(\finiteField\right)$ be a regular elliptic element with eigenvalues $\{\xi, \xi^q, \dots, \xi^{q^{a-1}}\}$, where $\xi \in \multiplicativegroup{\finiteFieldExtension{a}}$. Let $\omega_1, \dots, \omega_k$ be the roots of $\ExoticKloosterman^{\finiteFieldExtension{a}}\left(\alpha, \fieldCharacter\right)$ at $\xi$, as defined in the previous section. Then by \eqref{eq:power-sum-polynomial-m}, we have that for any partition $\lambda = \left(b_1,\dots,b_t\right) \vdash b$, $$p_{\lambda}\left(\omega_1,\dots,\omega_k\right) = \left(-1\right)^{\left(k-1\right) \ell\left(\lambda\right)} \prod_{j=1}^t \ExoticKloosterman_{a b_j}\left(\alpha, \fieldCharacter, \xi\right).$$
By \Cref{thm:matrix-kloosterman-sum-at-jordan-using-green-functions}, for any partition $\mu \vdash b$ $$\ExoticKloosterman\left(\alpha, \fieldCharacter, J_{\mu}\left(x\right) \right) = \left(-1\right)^{\left(k-1\right)n} q^{\left(k-1\right)\binom{n}{2}} \sum_{\lambda \vdash b}  \frac{Q_{\lambda}^{\mu}\left(q^a\right)}{z_{\lambda}} p_{\lambda}\left(\omega_1,\dots,\omega_k\right). $$
Therefore, we get from \eqref{eq:hall-littlewood-polynomial-in-terms-of-power-sum-polynomials}, the following reformulation of \Cref{thm:matrix-kloosterman-sum-at-jordan-using-green-functions}.
\begin{theorem}\label{thm:explicit-expression-for-jordan-block-with-hall-littlewood-polynomial}For any partition $\mu \vdash b$,
	$$\ExoticKloosterman\left(\alpha, \fieldCharacter, J_{\mu}\left(x\right) \right) = \left(-1\right)^{\left(k-1\right)n} q^{{\left(k-1\right)\binom{n}{2}}} \cdot \mathrm{\tilde{H}}_{\mu}\left(\omega_1,\dots,\omega_k;q^{a}\right).$$
\end{theorem}

When $\mu = \left(b\right)$, that is, when $J_{\mu}\left(x\right)$ is a regular element, we have that $$\mathrm{\tilde{H}}_{\mu}\left(x_1,\dots,x_k;t\right) = h_b\left(x_1,\dots,x_k\right),$$ where $$h_b\left(x_1,\dots,x_k\right) = \sum_{\lambda \vdash b} m_{\lambda}\left(x_1,\dots,x_k\right)$$ is the complete homogeneous symmetric polynomial of degree $b$. In this special case, we have the following corollary. It was proved in \cite[Section 5E]{erdelyi2021matrix} for the special case $a=1$, $k = 2$ and $\alpha = 1$.
\begin{corollary}\label{cor:symmetric-powers}
	$$\ExoticKloosterman\left(\alpha, \fieldCharacter, J_{(b)}\left(x\right) \right) = \left(-1\right)^{\left(k-1\right)n} 	q^{{\left(k-1\right)\binom{n}{2}}} \cdot \trace\left(\Frobenius \mid_{\xi}, \operatorname{Sym}^b \ExoticKloosterman^{\finiteFieldExtension{a}}\left(\alpha, \fieldCharacter\right)_{\xi}\right).$$
\end{corollary} 
As before, using \Cref{thm:multiplicativity-property} we are able to deduce a formula for a twisted matrix Kloosterman sum of a general element in terms of modified Hall--Littlewood polynomials.
\begin{theorem}\label{thm:hall-littlewood-formula-for-general-elements}
	Let $y_1 \in \GL_{a_1}\left(\finiteField\right)$, $\dots$, $y_s \in \GL_{a_s}\left(\finiteField\right)$ be regular elliptic elements with mutually disjoint eigenvalues over $\algebraicClosure{\finiteField}$. Let $\mu_1, \dots, \mu_s$ be partitions, such that $\sum_{j=1}^s a_j \abs{\mu_j} = n$. For any $j$, let $\{ \eta_j, \eta_j^q, \dots, \eta_j^{q^{a_j-1}} \}$ be the eigenvalues of $y_j$, where $\eta_j \in \multiplicativegroup{\finiteFieldExtension{a_j}}$, and let $\omega_{j1},\dots,\omega_{jk}$ be the roots of $\ExoticKloosterman^{\finiteFieldExtension{a_j}}\left(\alpha, \fieldCharacter\right)$ at $\eta_j$. Denote $y = \diag\left(J_{\mu_1}\left(y_1\right),\dots,J_{\mu_s}\left(y_s\right)\right)$. Then
	$$ \ExoticKloosterman\left(\alpha, \fieldCharacter, y \right) = \left(-1\right)^{\left(k-1\right)n} q^{\left(k-1\right)\binom{n}{2}} \prod_{j=1}^s \mathrm{\tilde{H}}_{\mu_j}\left(\omega_{j1}, \dots, \omega_{jk} ; q^{a_j}\right).$$
\end{theorem}
In the special case where $\mu_1 = \left(b_1\right)$, $\dots$, $\mu_s = \left(b_s\right)$, we get a formula for twisted matrix Kloosterman sums of regular elements of $\GL_n\left(\finiteField\right)$ in terms of symmetric powers of twisted Kloosterman sheaves.
\begin{corollary}
	If $y_1, \dots, y_s$ are as in \Cref{thm:hall-littlewood-formula-for-general-elements} and $\sum_{j=1}^s a_j b_j = n$, that is, $y = \diag\left(J_{\left(b_1\right)}\left(y_1\right), \dots, J_{\left(b_s\right)}\left(y_s\right)\right),
	$ is a regular element, then $$ \ExoticKloosterman\left(\alpha, \fieldCharacter, y \right) = \left(-1\right)^{\left(k-1\right)n} q^{\left(k-1\right)\binom{n}{2}} \prod_{j=1}^s \trace\left(\Frobenius \mid_{\eta_j}, \operatorname{Sym}^{b_j} \ExoticKloosterman^{\finiteFieldExtension{a_j}}\left(\alpha, \fieldCharacter\right)_{\eta_j} \right).$$
\end{corollary}
Finally, as a corollary of \Cref{thm:explicit-expression-for-jordan-block-with-hall-littlewood-polynomial}, and of the purity result of twisted Kloosterman sheaves, we obtain upper bounds for twisted matrix Kloosterman sums.
\begin{corollary}\label{cor:upper-bound}
	\begin{enumerate}
		\item For any partition $\mu \vdash b$,
		$$\abs{\ExoticKloosterman\left(\alpha, \fieldCharacter, J_{\mu}\left(x\right) \right)} \le q^{\left(k-1\right) \frac{n^2}{2}} \cdot  \# {\left\{ \mathcal{F} \text{ is a weak flag in } \finiteFieldExtension{a}^b \text{ of length } k \mid J_{\mu}\left(\xi\right) \mathcal{F} = \mathcal{F} \right\}}.$$
		\item In the special case where $\mu = \left(b\right)$, we have
		$$\abs{\ExoticKloosterman\left(\alpha, \fieldCharacter, J_{\left(b\right)}\left(x\right) \right)} \le q^{\left(k-1\right) \frac{n^2}{2}}  \cdot \binom{b + k - 1}{b}.$$
	\end{enumerate}
\end{corollary}
\begin{proof}
	Using \eqref{eq:sum-of-monomials-with-fixed-flags-coefficients} and \cite[Proposition 2.3]{Ram2023}, we can rewrite the sum in \Cref{thm:explicit-expression-for-jordan-block-with-hall-littlewood-polynomial} as
	$$ \ExoticKloosterman\left(\alpha, \fieldCharacter, J_{\mu}\left(x\right) \right) = \left(-1\right)^{\left(k-1\right)n} q^{{\left(k-1\right)\binom{n}{2}}} \cdot \sum_{\left(b_1,\dots,b_k\right)} \sizeof{\mathcal{X}_{\left(b_1,\dots,b_k\right)}^{\finiteFieldExtension{a}}\left(J_\mu\left(\xi\right)\right)} \omega_1^{b_1} \dots \omega_k^{b_k},$$
	where the sum goes over all weak compositions $\left(b_1,\dots,b_k\right)$ of $b$ (of length $k$).
	
	Since $\abs{\omega_1} = \dots = \abs{\omega_k} = q^{\frac{a \left(k-1\right)}{2}}$, we have that for any weak composition $\left(b_1,\dots,b_k\right)$ of $b$, $$\abs{\omega_1^{b_1} \dots \omega_k^{b_k}} = q^{\frac{ab \left(k-1\right)}{2}} = q^{\frac{n \left(k-1\right)}{2}}.$$
	Therefore, by the triangle inequality,
	$$ \abs{\ExoticKloosterman\left(\alpha, \fieldCharacter, J_{\mu}\left(x\right) \right)} \le q^{\left(k-1\right)\left( \binom{n}{2} + \frac{n}{2} \right)} \sum_{\left(b_1,\dots,b_k\right)} \sizeof{\mathcal{X}_{\left(b_1,\dots,b_k\right)}^{\finiteFieldExtension{a}}\left(J_{\mu}\left(\xi\right)\right)},$$
	which yields the first part.
	
	For the second part, if $\mu = \left(b\right)$, then there exists a unique $J_{\left(b\right)}\left( \xi \right)$-invariant subspace of $\finiteFieldExtension{a}^b$ of every given dimension $\le b$. Hence, the number of weak flags of length $k$ fixed by $J_{\left(b\right)}\left( \xi \right)$ is the number of weak compositions $\left(b'_1,\dots,b'_k\right)$ of $b$, which is $\binom{b + k - 1}{k-1}$.
\end{proof}
As before, we may use \Cref{thm:multiplicativity-property} to obtain an upper bound for twisted matrix Kloosterman sums of any element.
\begin{corollary}
	Keep the notation of \Cref{thm:hall-littlewood-formula-for-general-elements}. Let $y = \diag\left(J_{\mu_1}\left(y_1\right), \dots, J_{\mu_s}\left(y_s\right)\right)$. Then
		\begin{equation*}
				\abs{\ExoticKloosterman\left(\alpha, \fieldCharacter, y \right)} \le q^{\left(k-1\right) \frac{n^2}{2}} \prod_{j=1}^s \# {\left\{ \mathcal{F}_j \text{ is a weak flag in } \finiteFieldExtension{a_j}^{b_j} \text{ of length } k \mid J_{\mu_j}\left(\eta_j\right) \mathcal{F}_j = \mathcal{F}_j \right\}}.
		\end{equation*}
	In the special case where $\mu_j = \left(b_j\right)$, that is, $y = \diag\left(J_{\left(b_1\right)}\left(y_1\right), \dots, J_{\left(b_s\right)}\left(y_s\right)\right)$ is a regular element, we have
	$$\abs{\ExoticKloosterman\left(\alpha, \fieldCharacter, y \right)} \le q^{\left(k-1\right) \frac{n^2}{2}} \cdot \prod_{j=1}^s \binom{b_j + k - 1}{b_j}.$$
	In particular, if $b_1 = \dots = b_s = 1$, that is, $y = \diag\left(y_1, \dots, y_s\right)$ is a regular semisimple element, then
	$$ \abs{\ExoticKloosterman\left(\alpha, \fieldCharacter, y \right)} \le q^{\left(k-1\right)\frac{n^2}{2}} \cdot k^s.$$
\end{corollary}

\appendix

\section{Relation to Speh representations}\label{appendix:speh-representations}

In this appendix, we explain the relation between twisted matrix Kloosterman sums and Speh representations.

Let $$U_k = N_{\left(1^k\right)} = \left\{\begin{pmatrix}
	1 & a_1 & \ast & \ast & \ast \\
	& 1 & a_2 & \ast & \ast \\
	& & \ddots & \ddots & \ast \\
	& &  & 1 & a_{k-1}\\
	& & & & 1
\end{pmatrix}\right\}.$$
Let $\fieldCharacter \colon \finiteField \to \multiplicativegroup{\cComplex}$ be a non-trivial character. We define a character $\fieldCharacter \colon U_k \to \multiplicativegroup{\cComplex}$ by the formula
$$\fieldCharacter\begin{pmatrix}
	1 & a_1 & \ast & \ast & \ast \\
	& 1 & a_2 & \ast & \ast \\
	& & \ddots & \ddots & \ast \\
	& &  & 1 & a_{k-1}\\
	& & & & 1
\end{pmatrix} = \fieldCharacter\left(\sum_{i=1}^{k-1} a_i\right).$$
A representation $\tau$ of $\GL_k\left(\finiteField\right)$ is called \emph{generic} if there exists $0 \ne v \in \tau$ such that $\tau\left(u\right)v = \fieldCharacter\left(u\right)v$, for every $u \in U_k$. Such $v$ is called a \emph{$\fieldCharacter$-Whittaker vector}. A well known result of Gelfand--Graev is that if $\tau$ is irreducible and generic, then its Whittaker vector is unique, up to scalar multiplication \cite[Corollary 5.6]{SilbergerZink00}.

It is well-known that for every choice of cuspidal representation $\sigma_1$, $\dots$, $\sigma_s$ of $\GL_{k_1}\left(\finiteField\right)$, $\dots$, $\GL_{k_s}\left(\finiteField\right)$, respectively, such that $k_1 + \dots + k_s = k$, there exists a unique irreducible generic representation $\tau$ of $\GL_k\left(\finiteField\right)$ with cuspidal support $\left\{\sigma_1,\dots,\sigma_s\right\}$ \cite[Theorem 5.5]{SilbergerZink00}. The parameter of this representation $\tau$ can be described as follows: let $\sigma'_1, \dots, \sigma'_{l'}$ be all the different cuspidal representations in the cuspidal support of $\tau$, where $\sigma'_j$ is an irreducible cuspidal representation of $\GL_{k'_j}\left(\finiteField\right)$. Let $m_j$ be the number of times that $\sigma'_j$ appears in the cuspidal support of $\tau$. Then $\tau$ has parameter $\varphi_{\tau}$, such that $\varphi_{\tau}\left(k'_j, \sigma'_j\right) = \left(m_j\right)$ for every $1 \le j \le l'$, and $\left(\right)$ outside of $\left(k'_j,\sigma'_j\right)_{j=1}^{l'}$.

For any irreducible generic representation $\tau$ of $\GL_{k}\left(\finiteField\right)$ and any $n \ge 1$, we may associate an irreducible representation of $\GL_{kn}\left(\finiteField\right)$ called the \emph{Speh representation} and denoted $\SpehRepresentation{\tau}{n}$. If $\tau$ is an irreducible generic representation, then, as explained above, it has a parameter $\varphi_{\tau}$, such that $\varphi_{\tau}\left(k_j, \sigma_j\right) = \left(m_j\right)$ where $m_j \ge 1$, for any irreducible cuspidal representation $\sigma_j$ of $\GL_{k_j}\left(\finiteField\right)$, such that $\sigma_j$ is in the cuspidal support of $\tau$. The Speh representation $\SpehRepresentation{\tau}{n}$ is the irreducible representation of $\GL_{kn}\left(\finiteField\right)$ whose parameter $\varphi_{\SpehRepresentation{\tau}{c}}$ is supported on $\left(k_j, \sigma_j\right)$ for every $\left(k_j, \sigma_j\right)$ in the support of $\varphi_{\tau}$, such that $$\varphi_{\SpehRepresentation{\tau}{c}}\left(k_j, \sigma_j\right) = \left(m_j, \dots, m_j\right) = \left(m_j^n\right),$$
where $\varphi_{\tau}\left(k_j, \sigma_j\right) = \left(m_j\right)$.

If $\alpha_1, \dots, \alpha_k \colon \multiplicativegroup{\finiteField} \to \multiplicativegroup{\cComplex}$ are characters, such that $\alpha_i \ne \alpha_j$ for $i \ne j$, then the parabolically induced representation $\tau = \alpha_1 \circ \dots \circ \alpha_k$ is an irreducible generic representation of $\GL_k\left(\finiteField\right)$. In this case, the Speh representation $\SpehRepresentation{\tau}{n}$ is the parabolically induced representation $\alpha_{1, \GL_n\left(\finiteField\right)} \circ \dots \circ \alpha_{k, \GL_n\left(\finiteField\right)}$, where $\alpha_{j, \GL_n\left(\finiteField\right)} \colon \GL_n\left(\finiteField\right) \to \multiplicativegroup{\cComplex}$ is the character $\alpha_{j, \GL_n\left(\finiteField\right)} = \alpha_j \circ \det$ (notice that $\SpehRepresentation{\tau}{n}$ is indeed an irreducible representation of $\GL_{nk}\left(\finiteField\right)$).

Let $\fieldCharacter_{\left(n^k\right)} \colon N_{\left(n^k\right)} \to \multiplicativegroup{\cComplex}$ be the following character of the unipotent radical $N_{\left(n^k\right)}$
$$\fieldCharacter_{\left(n^k\right)}\begin{pmatrix}
	\IdentityMatrix{n} & A_1 & \ast & \ast & \ast \\
	& \IdentityMatrix{n} & A_2 & \ast & \ast \\
	& & \ddots & \ddots & \ast \\
	& &  & \IdentityMatrix{n} & A_{k-1}\\
	& & & & \IdentityMatrix{n}
\end{pmatrix} = \fieldCharacter\left(\sum_{i=1}^{k-1} \trace A_i\right),$$
where $A_1,\dots,A_{k-1} \in \squareMatrix_n\left(\finiteField\right)$. In his master's thesis \cite{Carmon2023}, Oded Carmon proved that for any irreducible generic representation $\tau$ of $\GL_k\left(\finiteField\right)$, there exists a unique (up to scalar multiplication) vector $v \in \SpehRepresentation{\tau}{n}$, such that $\SpehRepresentation{\tau}{n}\left(u\right) v = \fieldCharacter_{\left(n^k\right)}\left(u\right) v$, for every $u \in N_{\left(n^k\right)}$. Such $v$ is called a \emph{$\left(k,n\right)$ $\fieldCharacter$-Whittaker vector}. Given this result, we may consider the matrix coefficient corresponding to $v$. Let $\innerproduct{\cdot}{\cdot}$ be an inner product on $\SpehRepresentation{\tau}{c}$ invariant under the $\GL_{kn}\left(\finiteField\right)$ action. Then we define for every $g \in \GL_{kn}\left(\finiteField\right)$, $$\besselSpehFunction{\tau}{n}\left(g\right) = \frac{\innerproduct{\SpehRepresentation{\tau}{n}\left(g\right)v}{v}}{\innerproduct{v}{v}},$$
where $0 \ne v \in \SpehRepresentation{\tau}{n}$ is a $\left(k,n\right)$ $\fieldCharacter$-Whittaker vector. This definition does not depend on the choice of the $\left(k,n\right)$ $\fieldCharacter$-Whittaker vector. We call $\besselSpehFunction{\tau}{n}$ the \emph{Bessel--Speh function} associated with the representation $\SpehRepresentation{\tau}{n}$ and the additive character $\fieldCharacter$.

In \cite[Section 5]{CarmonZelingher2024}, we express special values of the Bessel--Speh function in terms of exotic matrix Kloosterman sums. Let us write here our result for the special case where $\tau$ is a principal series representation. If $n=1$, we obtain a special case of a result of Curtis--Shinoda \cite{curtis2004zeta}.

\begin{theorem}
	Let $k > 1$ and let $\alpha_1, \dots, \alpha_k \colon \multiplicativegroup{\finiteField} \to \multiplicativegroup{\cComplex}$ be characters. Let $\tau$ be the unique irreducible generic subrepresentation of $\alpha_1 \circ \dots \circ \alpha_k$ (it is a representation of $\GL_k\left(\finiteField\right)$). Then for every $x \in \GL_n\left(\finiteField\right)$,
	
	$$ \besselSpehFunction{\tau}{n}\begin{pmatrix}
		& \IdentityMatrix{\left(k-1\right)n}\\
		x
	\end{pmatrix} = q^{-\left(k-1\right)n^2} \ExoticKloosterman\left(\alpha^{-1}, \fieldCharacter, \left(-1\right)^{k-1} x^{-1}\right),$$
where $\alpha^{-1} \colon \left(\multiplicativegroup{\finiteField}\right)^k \to \multiplicativegroup{\cComplex}$ is given by $\alpha^{-1} = \alpha_1^{-1} \times \dots \times \alpha_k^{-1}$, and $\ExoticKloosterman\left(\alpha^{-1}, \fieldCharacter, \left(-1\right)^{k-1} x^{-1}\right)$ is the twisted matrix Kloosterman sum (see \Cref{sec:twisted-kloosterman-sums}).
\end{theorem}

\bibliographystyle{abbrv}
\bibliography{references}
\end{document}